\documentclass[a4papes]{article}

\usepackage{amsmath}
\usepackage{a4wide}
\usepackage{graphicx}
\usepackage{arydshln}
\usepackage{cite}
\usepackage{tikz}
\usepackage{tikzscale}
\usetikzlibrary{matrix}
\newlength\figureheight
\newlength\figurewidth
\usepackage{pgfplots}
\usepackage{amssymb}
\usepackage{mathdots}
\usepackage{multirow}
\usepackage{algorithm}
\usepackage{algpseudocode}
\usepackage{pdfpages}
\usepackage{hyperref}
\usepackage[noblocks]{authblk}

\usepackage{verbatim}
\usepackage{tikz}
\usetikzlibrary{plotmarks}

\usepackage{subcaption}

\usepackage{url}
\usepackage{amsthm}
\newtheorem{theorem}{Theorem}[section]

\newtheorem{definition}{Definition}[section]
\newtheorem{problem}{Problem}[section]

\newtheorem{example}{Example}[section]

\begin{document}
	\title{An Arnoldi-based approach to polynomial and rational least squares problems}
	\providecommand{\keywords}[1]{\textit{Keywords: } #1}
	\date{\today}

	\author[1a*]{Amin Faghih}
	\author[1b]{Marc Van Barel}
	\author[2c]{Niel Van Buggenhout}
	\author[1d]{Raf Vandebril}
	\affil[1]{Department of Computer Science, KU Leuven, University of Leuven, Leuven, Belgium}
	\affil[2]{Department of Mathematics, Universidad Carlos III de Madrid, Madrid, Spain}

\affil[a*]{Corresponding author: Amin Faghih \texttt{amin.faghih@kuleuven.be}}
\affil[b]{Marc Van Barel \texttt{marc.vanbarel@kuleuven.be}}
\affil[c]{Niel Van Buggenhout \texttt{nvanbugg@math.uc3m.es}}
\affil[d]{Raf Vandebril \texttt{raf.vandebril@kuleuven.be}}

	\renewcommand*{\Affilfont}{\small\it}

	\maketitle	
\begin{abstract}
In this research, we solve polynomial, Sobolev polynomial, rational, and Sobolev rational least squares problems. Although the increase in the approximation degree allows us to fit the data better in attacking least squares problems, the ill-conditioning of the coefficient matrix fuels the dramatic decrease in the accuracy of the approximation at higher degrees. To overcome this drawback, we first show that the column space of the coefficient matrix is equivalent to a Krylov subspace. Then the connection between orthogonal polynomials or rational functions and orthogonal bases for Krylov subspaces in order to exploit Krylov subspace methods like Arnoldi orthogonalization is established. Furthermore, some examples are provided to illustrate the theory and the performance of the proposed approach.
\end{abstract}
\\
\begin{small} {\textbf{AMS subject classification:}} 41A10, 41A20, 65D15, 65D05, 65F25.
	\end{small}
	\renewcommand{\thefootnote}{\fnsymbol{footnote}}

	\section{Introduction}
	A basic problem in science is to fit a model to observations subject to errors. It is clear that the more observations that are available, the more accurately it will be possible to calculate the parameters in the model. This gives rise to the so-called least squares problem or the problem of solving an overdetermined linear or nonlinear system of equations. It can be shown that the solution which minimizes a weighted sum of the squares of the residual is optimal in a certain sense. Applications of least squares occur in many areas of applied and engineering research such as statistics, photogrammetry,
signal processing, and control \cite{1,2,3}. Three standard ways to solve the least squares problem are: the normal equations, the QR decomposition, and the singular value decomposition (SVD). These techniques can be found in several text books such as the one from Björck \cite{2}, Demmel \cite{1}, and Lawson and Hanson \cite{3}.

Least squares can be used to fit a set of data to the best fit polynomial of a specified degree. Let $\{z_j,f_j\}_{j=1}^m$ be the set of $m$ data points in question. Given $n < m-1$, the least squares solution
\[p(t)=\sum_{k=0}^{n}c_kt^{k},\]
is determined through solving the overdetermined system of equations $V\bold{c} \approx \bold{f}$, including the Vandermonde matrix $V$	\begin{equation*}
		\underbrace{\begin{bmatrix}
			1& z_1 & \dots & z_1^n\\
			1& z_2 & \dots & z_2^n\\
			\vdots& \vdots &  & \vdots\\
			1& z_m & \dots & z_m^n
		\end{bmatrix}}_{=:V} \begin{bmatrix}
			c_0\\
			c_1\\
			\vdots\\
			c_n
		\end{bmatrix} \approx \begin{bmatrix}
			f_1\\
			f_2\\
			\vdots\\
			f_m
		\end{bmatrix},
	\end{equation*}

	We expect the increase in the polynomial degree to let us fit the data better. However, due to the ill-conditioning of the Vandermonde matrix, a dramatic increase in the residual norm for a certain degree may happen \cite{1,2,3}. To tackle this difficulty, a good choice of basis functions, e.g., an orthonormal basis, can lead to better fits and less ill-conditioned systems \cite{2,3,12}.
	
	The connection between orthogonal polynomials and orthogonal bases for Krylov
subspaces \cite{4} allows the use of Krylov subspaces methods for solving least squares problems. A polynomial Krylov subspace is defined for a matrix
$A \in \mathbb{C}^{m \times m}$ and starting vector $\bold{v} \in \mathbb{C}^{m}$ as
\[\mathcal{K}_{n+1}(A, \bold{v}) := \operatorname{span}\{\bold{v}, A\bold{v}, A^{2} \bold{v}, . . . , A^{n} \bold{v}\}.\]
Krylov subspace methods have been used to solve the problem of constructing a sequence of orthogonal polynomials for a given inner product and a set of data \cite{5,6,7,8,9}.\\

Recently, Brubeck et al. \cite{10} exploited the connection between a Krylov subspace
method, the Arnoldi iteration, and a polynomial least squares problem directly, by
reinterpreting the columns of the Vandermonde matrix as a basis for a Krylov
subspace. Such an interpretation has also been used implicitly in the work by Reichel et al. \cite{Re91,ReAmGr91} for solving least squares problems. They proposed two procedures. The former uses the Arnoldi iteration
to obtain coefficients of the approximating polynomial in a basis of orthonormal
polynomials, thereby avoiding explicit construction of the, often ill-conditioned,
Vandermonde system. The latter reuses the recurrence matrix computed by the
Arnoldi iteration to evaluate the obtained approximating polynomial in a set of given nodes.\\

	In addition, Niu et al. \cite{11} handled the Hermite least squares problem, where the least squares criterion involves both
function values and the first derivative values, and they extended the Vandermonde with Arnoldi approach introduced by Brubeck et al. \cite{10} to dealing with the confluent Vandermonde matrix
	\begin{equation*}
	\begin{bmatrix}
	1 & z_1 & z_1^2 & \dots & z_1^{n}\\
	1 & z_2 & z_{2}^2 & \dots & z_{2}^{n}\\
	\vdots & \vdots & \vdots & & \vdots \\
	1 & z_m & z_{m}^2 & \dots & z_{m}^{n}\\
		0 & 1 & 2 z_1 & \dots & n z_1^{n-1}\\
		0 & 1 & 2 z_2 & \dots & n z_2^{n-1}\\		
				\vdots & \vdots & \vdots & & \vdots \\
				0 & 1 & 2 z_m & \dots & n z_{m}^{n-1}\\				
		\end{bmatrix}.
	\end{equation*}
		They proposed to consider the Krylov subspace generated by the block matrix
		\[A :=\begin{bmatrix}
		Z\\
		I & Z
	\end{bmatrix}, \quad Z = diag(\{z_{j}\}_{j=1}^{m})\in \mathbb{C}^{m\times m},\]
	and the block vector $\bold{v}:= \begin{bmatrix}
		\boldsymbol{1}\\
		\boldsymbol{0}
	\end{bmatrix}$.
	The Krylov matrix corresponding to the space $\mathcal{K}_n(A,\bold{v})$ is then a row permutation of the aforementioned confluent Vandermonde matrix.\\
	
	To solve the least squares problem formulated by Brubeck et al. \cite{10}, Vandermonde with Arnoldi uses as a starting vector a vector of all ones. In this paper, the connection between the orthonormal basis for a Krylov subspace generated by the Arnoldi iteration and a sequence of polynomials orthonormal with
respect to a specific weighted discrete inner product is first formalized. It is this connection which links a least squares problem to Krylov subspaces. This formalization allows for a straightforward generalization by considering the weights in the inner product as the elements of the starting vector. In fact, this enables us to extend Vandermonde with Arnoldi to solve the weighted polynomial least squares problem.
	
	Although Niu et al. \cite{11} tackled the Hermite least squares problem through the link to the Krylov subspaces, both the numerical examples and the detailed strategy are for the first derivative case, and the first derivative and the function values in all nodes are considered to exist. Moreover, the starting vector is again assumed to be vector of all ones. We however handle the weighted Sobolev least squares problem \cite{13} involving the derivative values up to an arbitrary order. This is related to the confluent Vandermonde matrix, and we show that it can be interpreted using a Krylov subspace generated by a Jordan-like matrix \cite{25}. The orthonormal basis for such a Krylov subspace is related to Sobolev orthogonal polynomials \cite{14}.

Rational Krylov subspaces \cite{16} have
received a lot of attention in recent years \cite{17,18,19,20,21,22,23}. We proceed by generalizing the Arnoldi approach to rational Arnoldi allowing us to solve rational, and Sobolev rational least squares problems. The same approach can be taken to connect the rational least squares problem with the rational Krylov subspaces. Indeed we exploit an orthonormal rational basis which results in a less ill-conditioned system than a system including a Cauchy matrix.
	
Moreover, a new family of rational functions orthogonal with respect to a weighted discrete Sobolev inner product \cite{14,25}, which involves not only the function values but also the derivative values up to a given order, is introduced. The rational Arnoldi iteration gives a Hessenberg pencil implying the recurrence relation of this novel class of orthogonal functions \cite{23,24}, and the Sobolev rational least squares problem is then tackled based on these functions. For this purpose, we link the main problem to a rational Krylov subspace generated by a Jordan-like matrix.

 We also clarify the connection to structured matrices \cite{30} for both the polynomial and rational case. Analyzing the displacement structure allows us to identify which least squares problems can be solved with Krylov subspace techniques.\\

	This paper is arranged as follows. In Section \ref{s2}, we establish the
connection between the basis for polynomial Krylov subspaces and orthogonal polynomials which allows a straightforward generalization to solving weighted polynomial least squares. Section \ref{s3} handles the
Sobolev least squares problem, and we find a Krylov subspace created by a Jordan-like matrix to generate the column space of the confluent Vandermonde matrix. Section \ref{s4} presents a rational Arnoldi to solve the rational least squares problem, where the approximation is sought in a space of rational functions. Section \ref{s5} is devoted to constructing the recurrence relation of Sobolev orthogonal rational functions through rational Arnoldi. The coefficient matrix arising from Sobolev rational least squares problem,
that also includes information on the derivatives, can be generated by a rational Krylov subspace involving a Jordan-like matrix. Each section includes two algorithms, one is dedicated to computing the recurrence coefficients, and the second one indicates the evaluation of approximation at a set of sampling nodes. The numerical experiments are presented in Section \ref{s6} and include a comparison to the techniques without employing Arnoldi.
	\section{Weighted polynomial least squares problem}\label{s2}
	We first start this section by introducing the following weighted polynomial least squares problem, and then continue by taking the approach of Vandermonde with Arnoldi \cite{10} to solve the problem.
	\begin{problem}[Weighted polynomial least squares]\label{prob:wLS}
		Given a set of distinct nodes and function values $\{z_j,f_j\}_{j=1}^m$ with corresponding weights $\{w_j\}_{j=1}^m$, construct a polynomial $p\in\mathcal{P}_n$, $n\leq m-1$, minimizing
		\begin{equation}\label{eq1}
			\sum_{j=1}^m \vert w_j\vert^2 \vert p(z_j)-f_j\vert^2,
		\end{equation}
		in which $\mathcal{P}_n$ is the space of all polynomials of degree at most $n$.
	\end{problem}
	The relation \eqref{eq1} can be rewritten as $\sum_{j=1}^m  \vert w_j p(z_j)- w_jf_j\vert^2$, and Problem \ref{prob:wLS} is equivalent to the system of equations $WV\bold{c}\approx W\bold{f}$
	\begin{equation}\label{eq2}
		W\begin{bmatrix}
			1& z_1 & \dots & z_1^n\\
			1& z_2 & \dots & z_2^n\\
			\vdots& \vdots &  & \vdots\\
			1& z_m & \dots & z_m^n
		\end{bmatrix} \begin{bmatrix}
			c_0\\
			c_1\\
			\vdots\\
			c_n
		\end{bmatrix} \approx W \begin{bmatrix}
			f_1\\
			f_2\\
			\vdots\\
			f_m
		\end{bmatrix},
	\end{equation}
	including the Vandermonde matrix $V$, and $W=diag(\{w_j\}_{j=1}^m)$.
Solving this system provides us with coefficients $\{c_k\}_{k=0}^n$ representing the least squares solution, i.e., $p(t) = \sum_{k=0}^n c_k t^k$. Here, we refuse to solve this system of equations directly, and instead of using an ill-conditioned monomial basis, we reformulate \eqref{eq2} into a problem based on an orthonormal basis.
	
It is straightforward to see that the coefficient matrix $WV$ is in fact the Krylov matrix
\begin{equation*}
		B_{n+1}= \begin{bmatrix}
 			\bold{v} & Z \bold{v} & \dots & Z^{n}\bold{v}
 		\end{bmatrix},
	\end{equation*}
associated with the polynomial Krylov subspace $\mathcal{K}_{n+1}(Z,\bold{v})$ with \begin{equation*}
Z = diag(\{z_{j}\}_{j=1}^{m})\in \mathbb{C}^{m\times m}, \quad \quad \text{and} \quad \quad \bold{v}=\begin{bmatrix}
		w_1 & \dots & w_m
	\end{bmatrix}^\top\in\mathbb{C}^m.
	\end{equation*}

	Now, it is time to formalize the connection between the orthonormal basis for the Krylov subspace generated by the Arnoldi iteration and a sequence of orthonormal polynomials. Indeed, this connection links the least squares problem to Krylov subspaces.
\subsection{Arnoldi and orthogonal polynomials}
	No breakdown will occur applying the Arnoldi procedure, since the columns of the Vandermonde matrix are linearly independent for a set of distinct nodes.
	As orthonormal bases are more suited to the numerical computation, one uses the Arnoldi iteration to compute a nested orthonormal basis $Q_{n+1} = \begin{bmatrix}
		\bold{q}_0 & \bold{q}_1 &\dots & \bold{q}_n
	\end{bmatrix}$, satisfying $B_{n+1} = Q_{n+1} R$, with an upper triangular matrix $R\in\mathbb{C}^{(n+1)\times (n+1)}$. In fact, the Arnoldi iteration generates the orthonormal basis $Q_{n+1}\in\mathbb{C}^{m\times (n+1)}$ and an upper-Hessenberg matrix $\underline{H}\in\mathbb{C}^{(n+1)\times n}$ containing the recurrence coefficients for the columns of $Q_{n+1}$, i.e., the equation $ZQ_{n} = Q_{n+1} \underline{H}$ holds.\\
	
Now, the question is: when does the recurrence matrix for a nested orthogonal basis for $\mathcal{K}_{n+1}(Z,\bold{v})$ contain recurrence coefficients for a sequence of orthogonal polynomials $\{p_{k}\}_{k=0}^{n}$? The following theorem relates vectors in Krylov subspaces with polynomials and states how the
inner products correspond.
	\begin{theorem}[Krylov induced orthogonal polynomials]\label{theorem:KrylovOPs}
		Consider the diagonal matrix $Z = diag(\{z_{j}\}_{j=1}^{m})\in \mathbb{C}^{m\times m}$ and the vector $\bold{v}=\begin{bmatrix}
		w_1 & \dots & w_m
	\end{bmatrix}^\top\in\mathbb{C}^m$.
		Let $\begin{bmatrix}
			\bold{q}_0 & \bold{q}_1 & \dots & \bold{q}_{n}
		\end{bmatrix} = Q_{n+1} \in\mathbb{C}^{m\times (n+1)}$ form a nested orthonormal basis for $\mathcal{K}_{n+1}(Z,\bold{v})$.
		Suppose that polynomials $p_k\in\mathcal{P}_k$ satisfy $\bold{q}_k = p_k(Z)\bold{v}$.
		Then, the sequence of polynomials $\{p_k\}_{k=0}^{n}$ consists of orthonormal polynomials with respect to the inner product
		\begin{equation}\label{eqinner}
			\langle p_k,p_h\rangle_m =  \sum_{j=1}^{m}\vert w_j\vert^2 p_k(z_j)\overline{p_h(z_j)}.
		\end{equation}
	\end{theorem}
	\begin{proof}
		This is a straightforward consequence of an existing study of Liesen and Strakoš \cite{4}, given in Section 3.7.4.
	\end{proof}

	Theorem \ref{theorem:KrylovOPs} reveals that an orthonormal basis for $\mathcal{K}_{n+1}(Z,\bold{v})$ corresponds to the set of orthonormal polynomials $\{p_k\}_{k=0}^{n}$. This is the key observation on which the Vandermonde with Arnoldi procedure is based. Indeed, this approach employs orthonormal polynomials as basis functions which are far superior to monomials \cite{2,3,12}.\\

By equation \eqref{eq2} and $WV=B_{n+1}$, we have
\begin{equation*}
 B_{n+1}\bold{c}=W\bold{f} \xrightarrow{\text{Arnoldi}} Q_{n+1}R\bold{c}=W\bold{f} \xrightarrow{\bold{y}:=R\bold{c}} Q_{n+1}\bold{y}=W\bold{f},
\end{equation*}
in which $\bold{y}=\begin{bmatrix}
		y_0 & \dots & y_n
	\end{bmatrix}^\top\in\mathbb{C}^{n+1}$ contains the coefficients of $p(t)$ in a basis of the orthonormal polynomials. Equation \eqref{eq2} is now solved in its equivalent form $Q_{n+1}\bold{y}=W\bold{f}$.
	
According to Theorem \ref{theorem:KrylovOPs}, applying the Arnoldi iteration to $\mathcal{K}_{n+1}(Z,\bold{v})$ results in $Q_{n+1}$ satisfying
\begin{equation*}
		\underbrace{W\begin{bmatrix}
			p_0(z_1)& p_1(z_1) & \dots & p_n(z_1)\\
			p_0(z_2)& p_1(z_2) & \dots & p_n(z_2)\\
			\vdots& \vdots &  & \vdots\\
			p_0(z_m)& p_1(z_m) & \dots & p_n(z_m)\\
		\end{bmatrix}}_{Q_{n+1}} \begin{bmatrix}
			y_0\\
			y_1\\
			\vdots\\
			y_n
		\end{bmatrix} \approx W \begin{bmatrix}
			f_1\\
			f_2\\
			\vdots\\
			f_m
		\end{bmatrix}.
	\end{equation*}
	The Arnoldi iteration also gives us a Hessenberg recurrence matrix $\underline{H}$ generating the orthonormal polynomial sequence $\{p_{k}\}_{k=0}^{n}$ via
	
	\begin{equation}\label{eq3}
		t \begin{bmatrix}
			p_0 & p_1 & \dots & p_{n-1}
		\end{bmatrix} = \begin{bmatrix}
		p_0 & p_1 & \dots & p_{n-1} & p_n
	\end{bmatrix} \underline{H}.
	\end{equation}
	Finally, the vector of unknowns $\bold{y}$ can be calculated through $\bold{y}=Q_{n+1}^H W\bold{f}$, and thereby the least squares solution shall be $p(t)=\sum_{k=0}^{n} y_{k} p_{k}(t)$.\\
	
	A small adjustment to \verb|polyfitA| proposed by Brubeck et al. \cite{10} allows us to include weights and enables us to compute the vector of unknowns $\bold{y}$ through Algorithm \ref{alg:1}.
	\begin{algorithm}[H]
		\caption{The evaluation process of the unknown vector $\bold{y}$}\label{alg:1}
		\begin{algorithmic}[1]
			\State \textbf{Input:} $Z \in \mathbb{C}^{m\times m}$, $\bold{f} \in \mathbb{C}^{m}$, $\bold{v}\in\mathbb{C}^m$, the integer $n$
			\State \textbf{Output:} Hessenberg matrix $\underline{H}\in\mathbb{C}^{(n+1)\times n}$, the unknown vector $\bold{y} \in \mathbb{C}^{n+1}$
			\Procedure{}{}
			\State $\bold{q}_0 = \bold{v}/ \Vert \bold{v}\Vert$\Comment{Arnoldi iteration}
			\For{$k=1,2,\dots,n$}
			\State $\bold{q}_k = Z \bold{q}_{k-1}$
			\For{$j=1,2,\dots, k$}\Comment{Orthogonalization}
			\State $h_{j,k} = \langle \bold{q}_k,\bold{q}_j\rangle_E$
			\State $\bold{q}_k = \bold{q}_k - h_{j,k} \bold{q}_j$
			\EndFor
			\State $h_{k+1,k} = \Vert \bold{q}_k\Vert$
			\State $\bold{q}_k = \bold{q}_k/h_{k+1,k}$ \Comment{Normalization}
			\EndFor
			\State $\bold{y} = Q_{n+1}^H \textrm{diag}(w_1,\dots,w_m)\bold{f}$\Comment{With $\bold{v}=\begin{bmatrix}
					w_1 & \dots & w_m
				\end{bmatrix}^\top\in\mathbb{C}^m$}
			\EndProcedure\end{algorithmic}
	\end{algorithm}	
	Note that we have chosen to use the normalization $Q_{n+1}^H Q_{n+1} = I$, which is different from the normalization in \cite{10}, where they used a scaled unitary matrix.
	
	For general complex nodes $\{z_j\}_j$, the Arnoldi method performs well and the Hessenberg matrix encodes a long recurrence relation for orthogonal polynomials. For nodes on the real line $z_j\in \mathbb{R}$ or on the unit circle $z_j = e^{i \theta_j}$, with $\theta_j\in [0,2\pi]$, there exist short recurrence relations \cite{Sz75} and, therefore, the standard Arnoldi iteration performs unnecessary computations. A straightforward specialization of the Arnoldi iteration for real nodes suggests the Lanczos iteration which generates a tridiagonal matrix, however the Lanczos iteration is sensitive to rounding errors.
	Two algorithms that exploit the short recurrence relation and are robust to rounding errors are those of Gragg and Harrod \cite{9}, and Laurie \cite{La99}. Reichel applied the former to solve least squares problems involving real nodes \cite{Re91}.
	For nodes on the unit circle, an efficient algorithm exploiting the short recurrence relations has also been proposed \cite{ReAmGr91}.
	
	In view of the structure of $\underline{H}$, we can exploit the recurrence relation \eqref{eq3} and evaluate the least squares solution $p(t)$ at an arbitrary set of points $\{x_{j}\}_{j=1}^{M}$ through the underlying algorithm. In Algorithm \ref{alg:2}, the same process as Algorithm \ref{alg:1} is applied with a new matrix $X =diag(\{x_{j}\}_{j=1}^{M})\in \mathbb{C}^{M\times M}$, and we call the resulting matrix
		\begin{equation*}
	U_{n+1}:=	\begin{bmatrix}
					p_0(x_1)& p_1(x_1) & \dots & p_n(x_1)\\
					p_0(x_2)& p_1(x_2) & \dots & p_n(x_2)\\
					\vdots& \vdots &  & \vdots\\
					p_0(x_M)& p_1(x_M) & \dots & p_n(x_M)\\
			\end{bmatrix},
		\end{equation*}	
		satisfying
		\begin{equation*}
		XU_{n} = U_{n+1}\underline{H}.
		\end{equation*}
		Now, the vector $\bold{p}$ of order $M$ including $p(x_{j})$ is obtained if $U_{n+1}\bold{y}$ is evaluated. Meanwhile, the orthonormal polynomial of degree zero $p_{0}(t)=\frac{1}{\sqrt{w_{1}^{2}+...+w_{m}^{2}}}$ can be directly computed through the inner product \eqref{eqinner}.
			\begin{algorithm}[H]
			\caption{The least squares solution $\bold{p}$}\label{alg:2}
			\begin{algorithmic}[1]
				\State \textbf{Input:} $\bold{y} \in \mathbb{C}^{n+1}$, $\underline{H}\in\mathbb{C}^{(n+1)\times n}$, $X \in \mathbb{C}^{M \times M}$, $p_{0}(t)$
				\State \textbf{Output:} $U_{n+1} \in \mathbb{C}^{M\times (n+1)}$, $\bold{p} \in \mathbb{C}^{M}$
				\Procedure{}{}
				\State $\bold{u}_0 =\bold{p}_{0}$\Comment{With $\bold{p}_{0}$ as a $M$-vector of all $p_{0}(t)$}
				\For{$k=1,2,\dots,n$}
				\State $\bold{u}_k = X \bold{u}_{k-1}$
				\For{$j=1,2,\dots, k$}
				\State $\bold{u}_k = \bold{u}_k - h_{j,k} \bold{u}_j$
				\EndFor
				\State $\bold{u}_k = \bold{u}_k/h_{k+1,k}$ \Comment{Computing $\{p_{k}\}_{k=1}^{n}$ at sampling points via the recurrence relation \eqref{eq3}}
				\EndFor
				\State $\bold{p} = U_{n+1}\bold{y} $\Comment{The least squares solution}
				\EndProcedure\end{algorithmic}
		\end{algorithm}	
In the next subsection, by analyzing the displacement structure of the Vandermonde matrix, we will connect polynomial least squares problems to Krylov subspace methods.
	\subsection{Displacement rank and polynomial Krylov subspaces}\label{2.4}
	Any Krylov matrix associated with a Krylov subspace generated by a diagonalizable matrix is related to a Vandermonde matrix and vice versa.
	A Vandermonde matrix $V$ is characterized by its displacement rank \cite{30,26}.
	
	For the displacement operator based on $Z = \textrm{diag}(\{z_j\}_{j=1}^m)\in\mathbb{C}^{m\times m}$ and the left-shift matrix $S_{n+1} \in\mathbb{C}^{(n+1)\times (n+1)}$, we can write
	\begin{equation*}
		Z V - V S_{n+1} = \begin{bmatrix}
			& z_1^{n+1}\\
			0 & \vdots\\
			& z_m^{n+1}
		\end{bmatrix}.
	\end{equation*}
	This means that the displacement rank of the Vandermonde matrix $V$ is equal to $1$ for this specific displacement operator.\\

	This implies that the premultiplication of $V$ with $Z$ generates the next column in the matrix.
	This is the essential property of Krylov subspaces. They are generated by repeated premultiplication of a vector by a constant matrix.\\
	In the other direction, any Krylov matrix is related to a Vandermonde matrix and has displacement rank equal to 1. The following derivation reveals this relation for diagonalizable matrices.
	Substitute the eigenvalue decomposition $A= X Z X^{-1}$ in $B_{n+1}$ and set $\bold{w}:=X^{-1}\bold{v}$. Then, we have
	\begin{align*}
		B_{n+1}\ &= \begin{bmatrix}
			\bold{v} & A \bold{v} & \dots & A^{n}\bold{v}\\
		\end{bmatrix}
		= \begin{bmatrix}
			\bold{v} & X Z X^{-1} \bold{v} & \dots & (X Z X^{-1})^{n}\bold{v}\\
		\end{bmatrix}\\
		&=X \begin{bmatrix}
			w_1 & z_1 w_1 & \dots & z_1^{n}w_1\\
			w_2 & z_2 w_2 & \dots & z_2^{n}w_2\\
			\vdots & \vdots &  & \vdots\\
			w_m & z_m w_m & \dots & z_m^{n}w_m\\
		\end{bmatrix}\\
		&= X \underbrace{\begin{bmatrix}
				w_1 \\
				& w_2\\
				& & \ddots \\
				& & & w_m
		\end{bmatrix} }_{=:W}
		\underbrace{\begin{bmatrix}
				1 & z_1 & \dots & z_1^{n}\\
				1 & z_2 & \dots & z_2^{n}\\
				\vdots & \vdots &  & \vdots\\
				1 & z_m & \dots & z_m^{n}\\
		\end{bmatrix}}_{=:V},
	\end{align*}
	and clearly $\textrm{rank}(A B_{n+1} -B_{n+1} S_{n+1}) = 1$.
	In fact, any pair of matrices $A,B_{n+1}$ satisfying this equality are suitable candidates to apply an Arnoldi iteration to them. Furthermore, the Jordan canonical form $A = X J X^{-1}$ can be used for defective matrices, and this will be related to a confluent Vandermonde matrix, which is discussed in the next section.
	\section{Weighted Sobolev polynomial least squares problem}\label{s3}
	Another interesting case can be investigated by considering a Jordan-like matrix instead of a diagonal matrix to generate the Krylov subspace.
	This is related to the weighted Sobolev least squares problem which is solved by taking a confluent Vandermonde with the Arnoldi strategy.
	\begin{problem}[Weighted Sobolev polynomial least squares]\label{prob:wHLS}
		Let $\{z_j,w_{j}\}_{j=1}^\sigma$ be the given nodes and corresponding weights, respectively. Given the first $s_{j}$ derivative values $\{f^{(s_{j})}_{j}\}_{j=1}^\sigma$, $s_{j}\leq s$, with $s \geq 0$, construct a polynomial $p\in\mathcal{P}_{n}$ which minimizes
		\begin{equation*}
			\sum_{j=1}^\sigma \sum_{i=0}^{s_{j}} \vert w_j\vert^2  \left\vert \frac{\prod_{r=1}^{i}\alpha_{r}^{(j)}}{i!} \right\vert^2 \vert p^{(i)}(z_j)-f^{(i)}_{j}\vert^2,
		\end{equation*}
	\end{problem}
	where $\alpha_{r}^{(j)}$ are non-zero and chosen freely\footnote{$\alpha_{r}^{(j)}$ will appear in the Jordan blocks.}.
		
		Considering an arbitrary $s$, let us define matrices $\mathcal{B}_{j}\in \mathbb{C}^{(s_{j}+1)\times (n+1)}$, $j=1,2,\dots,\sigma$ as
	\begin{equation*}
	\begin{bmatrix}
	0 & \dots & \dots &0&s_{j}! & \dots &\frac{n!}{(n-s_{j})!}z_{j}^{n-s_{j}}\\
	0 &\dots& 0& (s_{j}-1)!  &s_{j}!z_{j} &\dots&\frac{n!}{(n-s_{j}+1)!}z_j^{n-s_{j}+1}\\
		\vdots & \iddots &  & \vdots& \vdots &&\vdots\\
		0 & 1 & \dots& (s_{j}-1)
		z_{j}^{s_{j}-2} &s_{j}
		z_{j}^{s_{j}-1}&\dots&nz_{j}^{n-1}\\
				1 & z_{j} & \dots& z_{j}^{s_{j}-1} &  z_{j}^{s_{j}}&\dots&z_{j}^{n}
\end{bmatrix},
\end{equation*}
	including the derivatives of monomials up to order $s_{j}$ at $z_{j}$. Now, Problem \eqref{prob:wHLS} is translated into the following system of equations consisting of a confluent Vandermonde matrix $V^{(c)}\in \mathbb{C}^{m\times (n+1)}$
	\begin{equation*}
		W\underbrace{\left[\begin{array}{ccccccc}
		&&&\mathcal{B}_{1}&&&\\
		&&&\mathcal{B}_{2}&&&\\
		&&&\vdots&&&\\
		&&&\mathcal{B}_{\sigma}&&&\\			
		\end{array}\right]}_{=:V^{(c)}} \begin{bmatrix}
			c_0\\
			c_1\\
			\vdots\\
			c_{n}
		\end{bmatrix} \approx W \bold{f},
	\end{equation*}
	in which the right-hand side vector $\bold{f} \in \mathbb{C}^{m}$ equals
	\begin{equation*}
\begin{bmatrix}
				f^{(s_{1})}_1~
				f^{(s_{1}-1)}_{1}~
			\dots~
			f_{1}&
			f^{(s_{2})}_2~
			f^{(s_{2}-1)}_{2}~
			\dots~
			f_{2}&
			\dots&
			f^{(s_{\sigma})}_\sigma~
			f^{(s_{\sigma}-1)}_{\sigma}~
			\dots~
			f_\sigma
		\end{bmatrix}^\top,
		\end{equation*}
	and $W=diag(\bold{w})$, the diagonal matrix of weights and factors compensating for the repeated differentiation which introduces factorial terms, that is,
	\begin{equation}\label{eq26}
	\bold{w}=\begin{bmatrix}     \frac{\prod_{r=1}^{s_{1}}\alpha_{r}^{(1)}}{s_{1}!}w_{1}  \ldots  \frac{\alpha_{1}^{(1)}}{1!} w_{1}& w_{1} &   \frac{\prod_{r=1}^{s_{2}}\alpha_{r}^{(2)}}{s_{2}!}w_{2}  \ldots  \frac{\alpha_{1}^{(2)}}{1!} w_{2}& w_{2}  & \dots&
		\frac{\prod_{r=1}^{s_{\sigma}}\alpha_{r}^{(\sigma)}}{s_{\sigma}!}w_{\sigma}  \ldots  \frac{\alpha_{1}^{(\sigma)}}{1!} w_{\sigma}& w_{\sigma}  \end{bmatrix}^\top .
		\end{equation}
		Here, $\bold{w} \in \mathbb{C}^{m}$, and $m:=\sigma+\sum_{j=1}^\sigma s_{j}$.
	
	Notably, matrix $W V^{(c)}$ is equal to the Krylov matrix for the Krylov subspace $\mathcal{K}(J,\bold{v})$, equipped with the starting vector
	\begin{equation*}
	\bold{v}=\begin{bmatrix}      \overbrace{0 \ldots  0}^{s_{1}} & w_{1} &   \overbrace{0 \ldots  0}^{s_{2}} & w_{2} & \dots& \overbrace{0 \ldots  0}^{s_{\sigma}} & w_{\sigma} \end{bmatrix}^\top \in \mathbb{C}^{m},
		\end{equation*}
		and the Jordan-like matrix
	\begin{equation*}
	 J = \left[\begin{array}{cccc}
		J_{s_{1}} &&&\\
		 &J_{s_{2}}&&\\
		 &&\ddots&\\
		&&&J_{s_{\sigma}}\\
	\end{array}\right]\in \mathbb{C}^{m\times m},
	\end{equation*}
	containing $\sigma$ blocks of order $s_{j}+1$
	\begin{equation}\label{eq25}
	J_{s_{j}} = \begin{bmatrix}
		z_{j}&\alpha_{s_{j}}^{(j)}&&\\
		 &z_{j}&&\\
		 &&\ddots&\alpha_{1}^{(j)}\\
		&&&z_{j}\\
		
	\end{bmatrix}, \quad j=1,2,\ldots,\sigma.
	\end{equation}
	
	As a consequence, a confluent Vandermonde with Arnoldi procedure can be developed.
	
	The next subsection is dedicated to a generalization of Algorithm \ref{alg:1} to solve Problem \ref{prob:wHLS}. We will introduce an inner product by which a set of Sobolev polynomials are orthonormal via Arnoldi.
	\subsection{Arnoldi and Sobolev orthogonal polynomials}\label{sec:Arnoldi_Sobolev}
	By taking the same approach as the previous section, we expect Arnoldi to generate an orthonormal basis $Q_{n+1}$ for the Krylov subspace $\mathcal{K}_{n+1}(J,\bold{v})$ and a Hessenberg recurrence matrix $\underline{H}$. Problem \ref{prob:wHLS} is called the Sobolev least squares problem, since the basis $Q_{n+1}$ is related to polynomials orthonormal with respect to a Sobolev inner product.
		\begin{theorem}[Krylov induced Sobolev orthogonal polynomials]\label{theorem:KrylovSobOPs}
		Consider the Jordan-like matrix \begin{equation*}
	 J = \left[\begin{array}{cccc}
		J_{s_{1}} &&&\\
		 &J_{s_{2}}&&\\
		 &&\ddots&\\
		&&&J_{s_{\sigma}}\\
	\end{array}\right]\in \mathbb{C}^{m\times m},
	\end{equation*}
	containing $\sigma$ blocks as defined in \eqref{eq25}, and the vector \begin{equation*}
	\bold{v}=\begin{bmatrix}      \overbrace{0 \ldots  0}^{s_{1}} & w_{1} &   \overbrace{0 \ldots  0}^{s_{2}} & w_{2} & \dots& \overbrace{0 \ldots  0}^{s_{\sigma}} & w_{\sigma} \end{bmatrix}^\top \in \mathbb{C}^{m}.
		\end{equation*}
		 Let $\begin{bmatrix}
			\bold{q}_0 & \bold{q}_1 & \dots & \bold{q}_{n}
		\end{bmatrix} = Q_{n+1} \in\mathbb{C}^{m\times (n+1)}$ form a nested orthonormal basis for $\mathcal{K}_{n+1}(J,\bold{v})$.
		Assume that polynomials $p_k\in\mathcal{P}_k$ satisfy $\bold{q}_k = p_k(J)\bold{v}$.
		Then $\{p_k\}_{k=0}^{n}$ is the sequence of Sobolev orthonormal polynomials with respect to the inner product
		\begin{equation}\label{innerS}
			\langle p_k,p_h\rangle _S =  \sum_{j=1}^{\sigma} \sum_{i=0}^{s_{j}}\vert w_j\vert^2 \left\vert \frac{\prod_{r=1}^{i}\alpha_{r}^{(j)}}{i!} \right\vert^2 p^{(i)}_{k}(z_j) \overline{p^{(i)}_{h}(z_j)}.
		\end{equation}
	\end{theorem}
	\begin{proof}
		The proof is done by Van Buggenhout \cite{25}, given in Section 3.2.
		\end{proof}
		
		To this end, the solution $p\in\mathcal{P}_n$ to Problem \ref{prob:wHLS} can be obtained in the well-conditioned basis $\{p_k\}_{k=0}^{n}$ as
	\begin{equation*}
		p(t) = \sum_{k=0}^n y_k p_k(t),
	\end{equation*}
	where the vector $\bold{y}=\begin{bmatrix}
		y_0 &  \dots & y_n
	\end{bmatrix}^{\top}$
	is the output of Algorithm \ref{alg:3}.
\begin{algorithm}[H]
	\caption{The evaluation process of the unknown vector $\bold{y}$}\label{alg:3}
	\begin{algorithmic}[1]
		\State \textbf{Input:} $J \in \mathbb{C}^{m\times m}$, $\bold{f} \in \mathbb{C}^{m}$, $W\in\mathbb{C}^{m \times m}$, $\bold{v}\in\mathbb{C}^m$, the integer $n$
		\State \textbf{Output:} Hessenberg matrix $\underline{H}\in\mathbb{C}^{(n+1)\times n}$, the unknown vector $\bold{y} \in \mathbb{C}^{n+1}$
		\Procedure{}{}
		\State $\bold{q}_0 = \bold{v}/ \Vert \bold{v}\Vert$\Comment{Arnoldi iteration}
		\For{$k=1,2,\dots,n$}
		\State $\bold{q}_k = J \bold{q}_{k-1}$
		\For{$j=1,2,\dots, k$}\Comment{Orthogonalization}
		\State $h_{j,k} = \langle \bold{q}_k,\bold{q}_j\rangle_E$
		\State $\bold{q}_k = \bold{q}_k - h_{j,k} \bold{q}_j$
		\EndFor
		\State $h_{k+1,k} = \Vert \bold{q}_k\Vert$
		\State $\bold{q}_k = \bold{q}_k/h_{k+1,k}$ \Comment{Normalization}
		\EndFor
		\State $\bold{y} = Q_{n+1}^H W\bold{f}$
		\EndProcedure\end{algorithmic}
\end{algorithm}	
	A set of sampling points $\{x_j\}_{j=1}^{\tau}$, in which to evaluate the least squares solution and the derivatives of the least squares solution of order at most $s$, can be input to Algorithm \ref{alg:4}. In this algorithm, we consider $\bold{S}$ as the column vector of sampling derivatives $\{\tilde{s}_{j}\}_{j=1}^{\tau}$, $\tilde{s}_{j} \leq s$, and $X$ is a Jordan-like matrix constructed by sampling nodes $\{x_j\}_{j=1}^{\tau}$.
	If we define the matrix $\mathcal{D}_{j} \in \mathbb{C}^{(\tilde{s}_{j}+1)\times (n+1)}$, $j=1,2,\ldots,\tau$, including the derivatives up to order $\tilde{s}_{j}$ of the Sobolev orthonormal polynomials at $x_{j}$, as follows
	\begin{equation*}
	\mathcal{D}_{j}:=	\begin{bmatrix}
			0  & 0& \dots&0&0&p_{\tilde{s}_{j}}^{(\tilde{s}_{j})}(x_{j})&\dots& p_{n}^{(\tilde{s}_{j})}(x_{j})\\
				0  &0& \dots &0& p_{\tilde{s}_{j}-1}^{(\tilde{s}_{j}-1)}(x_{j})&p_{\tilde{s}_{j}}^{(\tilde{s}_{j}-1)}(x_{j})&\dots& p_{n}^{(\tilde{s}_{j}-1)}(x_{j})\\
			\vdots & \vdots &  & \iddots& \vdots &\vdots&&\vdots\\
			
			0&0&\iddots&p_{\tilde{s}_{j}-2}^{\prime\prime}(x_{j}) &p_{\tilde{s}_{j}-1}^{\prime \prime}(x_{j})&p_{\tilde{s}_{j}}^{\prime \prime}(x_{j})&
			\dots& p_{n}^{\prime \prime}(x_{j})\\
			
			0 & p_{1}^{\prime}(x_{j}) & \dots&p_{\tilde{s}_{j}-2}^{\prime}(x_{j}) &p_{\tilde{s}_{j}-1}^{\prime}(x_{j})&p_{\tilde{s}_{j}}^{\prime}(x_{j}) &\dots& p_{n}^{\prime}(x_{j})\\
			p_{0}(x_{j}) & p_{1}(x_{j}) &\dots&p_{\tilde{s}_{j}-2}(x_{j})&p_{\tilde{s}_{j}-1}(x_{j})&p_{\tilde{s}_{j}}(x_{j}) &\dots & p_{n}(x_{j})
		\end{bmatrix},
	\end{equation*}
	the resulting matrix $U_{n+1} \in \mathbb{C}^{M\times (n+1)}$, $M:=\tau+\sum_{j=1}^\tau \tilde{s}_{j}$ is
	\begin{equation*}
	U_{n+1}:=	\left[\begin{array}{ccccc}
			&&\mathcal{D}_{1}&&\\
			&&\mathcal{D}_{2}&&\\
			&&\vdots&&\\
			&&\mathcal{D}_{\tau}&&\\
		\end{array}\right],
	\end{equation*}
	satisfying
	\begin{equation*}
		XU_{n} = U_{n+1}\underline{H}.
	\end{equation*}
	We can now compute the vector $\bold{p}:= U_{n+1}\bold{y}$ containing the approximations
	\begin{equation*}
		\bold{p}=\begin{bmatrix}     p^{(\tilde{s}_{1})}(x_{1}) \ldots  p^{\prime}(x_{1})& p(x_{1}) &   p^{(\tilde{s}_{2})}(x_{2}) \ldots  p^{\prime}(x_{2})& p(x_{2}) &  & \dots&
			p^{(\tilde{s}_{\tau})}(x_{\tau}) \ldots  p^{\prime}(x_{\tau})& p(x_{\tau})\end{bmatrix}^\top .
	\end{equation*}
	In addition, $p_{0}(t)=\frac{1}{\sqrt{w_{1}^{2}+...+w_{m}^{2}}}$ is the Sobolev orthonormal polynomial of degree zero computed through the inner product \eqref{innerS}.
		\begin{algorithm}[H]
		\caption{The least squares solution $\bold{p}$}\label{alg:4}
		\begin{algorithmic}[1]
			\State \textbf{Input:} $\bold{y} \in \mathbb{C}^{n+1}$, $\underline{H}\in\mathbb{C}^{(n+1)\times n}$, $X \in \mathbb{C}^{M \times M}$, $\tau$-vector $\bold{S}$, $p_{0}(t)$
			\State \textbf{Output:} $U_{n+1} \in \mathbb{C}^{M\times (n+1)}$, $\bold{p} \in \mathbb{C}^{M}$
			\Procedure{}{}
		\State $\bold{u}_0 = \begin{bmatrix}
			\mathcal{D}_{1,1}\\
			\mathcal{D}_{2,1}\\
			\vdots\\
			\mathcal{D}_{\tau,1}
		\end{bmatrix}$\Comment{	$\mathcal{D}_{j,1}$ is the first column of the matrix $\mathcal{D}_{j}$, for $j=1,2,\ldots,\tau$}
			\For{$k=1,2,\dots,n$}
			\State $\bold{u}_k = X \bold{u}_{k-1}$
			\For{$j=1,2,\dots, k$}
			\State $\bold{u}_k = \bold{u}_k - h_{j,k} \bold{u}_j$
			\EndFor
			\State $\bold{u}_k = \bold{u}_k/h_{k+1,k}$
			\EndFor
			\State $\bold{p} = U_{n+1}\bold{y}$\Comment{The least squares solution}
			\EndProcedure\end{algorithmic}
	\end{algorithm}	
	\section{Weighted rational least squares problem}\label{s4}
	The aforementioned stable approach to solve a polynomial least squares problem leads us to think how this strategy works for the rational case. In 1984, Ruhe \cite{16} introduced a generalization of polynomial Krylov subspaces to rational Krylov subspaces.
 	There are several essentially equivalent forms to define a rational Krylov subspace, the one used here is given as follows.
	\begin{definition}[Rational Krylov subspace \cite{27,16}]\label{def:rks}
		Let $A \in \mathbb{C}^{m\times m}$, $\bold{v}\in \mathbb{C}^m$,  $\Xi = \{\xi_k \}_{k=1}^{n}$ with $\xi_k=\frac{\mu_{k}}{\nu_{k}}\in\overline{\mathbb{C}}:= \mathbb{C} \cup \{ \infty \}$, and $\Phi = \{ \phi_k \}_{k=1}^n$ with $\phi_k= \frac{\rho_{k}}{\eta_{k}}\in \overline{\mathbb{C}}$.
		A rational Krylov subspace with poles $\xi_i$ and shifts $\phi_i$ is defined as
		\begin{equation*}
			\mathcal{K}_{n+1}(A,\bold{v};\Xi,\Phi):=  \operatorname{span}\left\{\bold{v}, \frac{\eta_{1} A-\rho_1 I}{\nu_{1} A-\mu_1 I}\bold{v},\dots , \prod_{k=1}^{n} \left(\frac{\eta_{k} A-\rho_{k} I}{\nu_{k} A-\mu_{k} I}\right) \bold{v} \right\}.
		\end{equation*}
		Here we have $\xi_{k} \ne \phi_{k}$, for $k=1,2,\dots,n$.
	\end{definition}	

Now let us define the weighted rational least squares problem.
	\begin{problem}[Weighted rational least squares]\label{prob:wCLS}
		Given data $\{z_j,f_j\}_{j=1}^m$, weights $\{w_j\}_{j=1}^m$, and the set of poles $\Xi$, construct a rational function $r\in\mathcal{R}_n^\Xi$, $n\leq m-1$, which minimizes
		\begin{equation*}
			\sum_{j=1}^m \vert w_j\vert^2 \vert r(z_j)-f_j\vert^2.
		\end{equation*}
		Here, $\mathcal{R}_n^\Xi:= \frac{\mathcal{P}_n}{{\underset{\xi_k \neq \infty }{\prod_{k=1}^{n}}(z-\xi_k)}}$.
	\end{problem}
	For finite poles, we can formulate Problem \ref{prob:wCLS} as an overdetermined system of equations involving matrix $W=diag(\{w_j\}_{j=1}^m)$ and $C\in\mathbb{C}^{m\times (n+1)}$
	\begin{equation}\label{Cauchy system}
		W\underbrace{\begin{bmatrix}
			1&\frac{1}{z_1-\xi_1} & \frac{1}{z_1-\xi_2} & \dots & \frac{1}{z_1-\xi_n}\\
			1&\frac{1}{z_2-\xi_1} & \frac{1}{z_2-\xi_2} & \dots & \frac{1}{z_2-\xi_n}\\
			\vdots & \vdots & & &\vdots \\
			1&\frac{1}{z_m-\xi_1} & \frac{1}{z_m-\xi_2} & \dots & \frac{1}{z_m-\xi_n}\\
		\end{bmatrix}}_{=:C}\begin{bmatrix}
	c_0\\
	c_1\\
	\vdots\\
	c_n
\end{bmatrix} \approx W\begin{bmatrix}
f_1\\
f_2\\
\vdots\\
f_m
\end{bmatrix},
 	\end{equation}
 	and solving this system results in the least squares solution $r(t) = c_{0}+\sum_{k=1}^{n} c_k \frac{1}{t-\xi_k}$.
 	
 	The basis $\{1,\frac{1}{t-\xi_1},\ldots, \frac{1}{t-\xi_n}\}$ of $\mathcal{R}_n^\Xi$ is related to the basis for the Krylov subspace $\mathcal{K}_{n+1}(Z,\bold{v};\Xi,\Phi)$ associated with the Krylov matrix
 	\begin{equation*}
 		B_{n+1} = \begin{bmatrix}
 			\bold{v} & \frac{ I}{Z-\xi_1 I} \bold{v} & \left(\frac{ I}{Z-\xi_1 I}\right) \left( \frac{Z-\xi_1 I}{Z-\xi_{2} I} \right)\bold{v}&\dots & \left(\frac{ I}{Z-\xi_1 I}\right)\prod_{k=1}^{n-1}\left( \frac{Z-\xi_k I}{Z-\xi_{k+1} I} \right)\bold{v}
 		\end{bmatrix} ,
 	\end{equation*}
 in which $\Phi = \{-\infty,\xi_1,\xi_2, \dots, \xi_{n-1} \}$, with $Z=diag(\{z_{j}\}_{j=1}^{m})\in \mathbb{C}^{m\times m}$ and
 \begin{equation*}
\bold{v}=\begin{bmatrix}
		w_1 & \dots & w_m
	\end{bmatrix}^\top\in\mathbb{C}^m.
	\end{equation*}

	A Cauchy matrix can be obtained by omitting the first column of $C$, and the same approach is also valid for this case. In this section we only consider distinct finite poles, which is related to Cauchy matrices. This connection between rational Krylov subspaces and Cauchy matrices allows a generalization of the Vandermonde with Arnoldi procedure which we call it Cauchy with Arnoldi. To do so, we require an Arnoldi-like iteration for rational Krylov subspaces. This is discussed in the next section.
	\subsection{Rational Arnoldi and orthogonal rational functions}
	 A procedure to compute an orthonormal basis for a rational Krylov subspace is due to Ruhe \cite{27} and is called the rational Arnoldi iteration. The following theorem states that an orthonormal basis can be generated by a recurrence relation that is governed by a Hessenberg pencil.
 	\begin{theorem}[Rational Arnoldi iteration \cite{27}]\label{theorem:RAI}
	 	Consider $A\in \mathbb{C}^{m\times m}$, $\bold{v}\in\mathbb{C}^m$, and the sets of poles $\Xi$ and shifts $\Phi$.
	 	Let $Q_{n+1}\in\mathbb{C}^{m\times (n+1)}$ be an orthonormal nested basis for the rational Krylov subspace $\mathcal{K}_{n+1}(A,\bold{v};\Xi,\Phi)$.
	 	Then $Q_{n+1}$ satisfies
	 	\begin{equation*}
	 		AQ_{n+1} \underline{K}= Q_{n+1} \underline{H},
	 	\end{equation*}
	 	for Hessenberg matrices $\underline{H},~ \underline{K}\in\mathbb{C}^{(n+1)\times n}$, with $\frac{(\underline{H})_{k+1,k}}{(\underline{K})_{k+1,k}} = \xi_k$, $k=1,2\dots, n$.
	 \end{theorem}
	 An orthonormal basis for a rational Krylov subspace can thus be generated by a recurrence relation that is captured by a Hessenberg recurrence pencil $(\underline{H}_n,\underline{K}_n)$ which is not unique\footnote{We should note that the pencil $(\underline{H}_n,\underline{K}_n)$ is not unique, i.e. a multiplication on the right with a non-singular upper triangular matrix is possible}. In order to make a connection between the orthonormal bases for rational Krylov subspaces and orthogonal rational functions the following theorem is presented.
	\begin{theorem}[Krylov induced orthogonal rational functions]\label{theorem:KrylovORFs}
		Consider the diagonal matrix $Z=diag(\{z_{j}\}_{j=1}^{m})\in \mathbb{C}^{m\times m}$ and the vector $
\bold{v}=\begin{bmatrix}
		w_1 & \dots & w_m
	\end{bmatrix}^\top\in\mathbb{C}^m
	$. Let $\begin{bmatrix}
			\bold{q}_0 & \bold{q}_1 & \dots & \bold{q}_{n}
		\end{bmatrix} = Q_{n+1} \in\mathbb{C}^{m\times (n+1)}$ form a nested orthonormal basis for $\mathcal{K}_{n+1}(A,\bold{v};\Xi,\Phi)$, $n<m$.
		Let rational functions $r_k\in\mathcal{R}^\Xi_k$ be defined on the spectrum of $Z$ satisfying $\bold{q}_k = r_k(Z)\bold{v}$.
		Then the sequence of rational functions $\{r_k\}_{k=0}^{n}$ consists of orthonormal rational functions with respect to an inner product of the form
		\begin{equation}\label{innerR}
			\langle r_k,r_h\rangle_m = \sum_{j=1}^{m} \vert w_j\vert^2  r_k(z_j)\overline{r_h(z_j)}.
		\end{equation}
	\end{theorem}
	\begin{proof}
	This is a direct result of Theorem 7.1 in \cite{24}.
	\end{proof}
	We proceed the same way as in Section \ref{s2}. Once we get hold of the Hessenberg recurrence pencil $(\underline{H}_n,\underline{K}_n)$ via the rational Arnoldi, the set of orthogonal rational functions $\{r_k\}_{k=0}^{n}$ can be obtained through
	\begin{equation*}
		t \begin{bmatrix}
			r_0 & r_1 & \dots & r_{n}
		\end{bmatrix}\underline{K} = \begin{bmatrix}
		r_0 & r_1 & \dots & r_{n}
	\end{bmatrix} \underline{H}.
	\end{equation*}
	Therefore, accessing $\bold{y}=Q_{n+1}^H W\bold{f}$ makes it possible to compute the least squares solution $r(t)=\sum_{k=0}^{n} y_{k} r_{k}(t)$.\\
	
	The unknown coefficients of $r(t)$ can be obtained by Algorithm \ref{alg:5} as follows.
		\begin{algorithm}[H]
			\caption{The evaluation process of the unknown vector $\bold{y}$}\label{alg:5}
			\begin{algorithmic}[1]
				\State \textbf{Input:} $Z \in \mathbb{C}^{m\times m}$, $\Xi = \{\xi_1,\xi_2,\dots, \xi_{n}\}$, $\Phi=\{\phi_1,\phi_2,\dots, \phi_{n}\}$, $\bold{f} \in \mathbb{C}^{m}$, $\bold{v} \in \mathbb{C}^{m}$, the integer $n$
				\State \textbf{Output:} Hessenberg matrices $\underline{H},~\underline{K} \in\mathbb{C}^{(n+1)\times n}$, the unknown vector $\bold{y} \in \mathbb{C}^{n+1}$
				\Procedure{}{}
				\State $\bold{q}_0 = \bold{v}/ \Vert \bold{v}\Vert$\Comment{Rational Arnoldi iteration}
				\For{$k=1,2,\dots,n$}
				\State $\bold{q}_k = (\nu_k Z - \mu_k I)^{-1} (\eta_k Z - \rho_k I)\bold{q}_{k-1}$ \Comment{With $\frac{\mu_k}{\nu_k} = \xi_k$, and $\frac{\rho_k}{\eta_k} = \phi_k$}
				\For{$j=1,2,\dots, k$}\Comment{Orthogonalization}
				\State $h_{j,k} = \langle \bold{q}_k,\bold{q}_j\rangle_E$
				\State $\bold{q}_k = \bold{q}_k - h_{j,k} \bold{q}_j$
				\EndFor
				\State $h_{k+1,k} = \Vert \bold{q}_k\Vert$
				\State $\bold{q}_k = \bold{q}_k/h_{k+1,k}$ \Comment{Normalization}
				\EndFor
				\State $\underline{K} =  \underline{H}\textrm{diag}(\nu_1,\dots,\nu_n) - I_{n+1}\textrm{diag}(\eta_1,\dots,\eta_n)$\Comment{$I_{n+1}\in \mathbb{C}^{(n+1)\times n}$ is the identity matrix}
				\State $\underline{H} =  \underline{H} \textrm{diag}(\mu_1,\dots,\mu_n) - I_{n+1}\textrm{diag}(\rho_1,\dots,\rho_n)$
				\State $\bold{y} = Q_{n+1}^H \textrm{diag}(w_1,\dots,w_m)\bold{f}$\Comment{With $\bold{v}=\begin{bmatrix}
						w_1 & \dots & w_m
					\end{bmatrix}^\top\in\mathbb{C}^m$}
				\EndProcedure\end{algorithmic}
		\end{algorithm}	
	 Given the first rational function $r_{0}(t)=\frac{1}{\sqrt{w_{1}^{2}+...+w_{m}^{2}}}$ evaluated through the inner product \eqref{innerR}, the solution can be computed by Algorithm \ref{alg:6} at the set of sampling points $\{x_{j}\}_{j=1}^{M}$. In this algorithm, matrix $U_{n+1}$ contains the values of the orthonormal rational functions at sampling points
	 \begin{equation*}
	 	U_{n+1}:=	\begin{bmatrix}
	 		r_0(x_1)& r_1(x_1) & \dots & r_n(x_1)\\
	 		r_0(x_2)& r_1(x_2) & \dots & r_n(x_2)\\
	 		\vdots& \vdots &  & \vdots\\
	 		r_0(x_M)& r_1(x_M) & \dots & r_n(x_M)\\
	 	\end{bmatrix},
	 \end{equation*}	
	 satisfying
	 \begin{equation*}
	 	XU_{n+1}\underline{K} = U_{n+1}\underline{H}, \quad  \quad X =diag(\{x_{j}\}_{j=1}^{M})\in \mathbb{C}^{M\times M}.
	 \end{equation*}
 If $\bold{r}:=U_{n+1}\bold{y}$ is now computed, the result is a vector of order $M$ including $r(x_{j})$.
 	\begin{algorithm}[H]
 	\caption{The least squares solution $\bold{r}$}\label{alg:6}
 	\begin{algorithmic}[1]
 		\State \textbf{Input:} $\bold{y} \in \mathbb{C}^{n+1}$, $\underline{H},~ \underline{K}\in\mathbb{C}^{(n+1)\times n}$, $X \in \mathbb{C}^{M \times M}$, $r_{0}(t)$
 		\State \textbf{Output:} $U_{n+1} \in \mathbb{C}^{M\times (n+1)}$, $\bold{r} \in \mathbb{C}^{M}$
 		\Procedure{}{}
 		\State $\bold{u}_0 =\bold{r}_{0}$\Comment{With $\bold{r}_{0}$ as a $M$-vector of all $r_{0}(t)$}
 		\For{$k=1,2,\dots,n$}
 		\State $\bold{u}_k = 0$
 		\For{$j=1,2,\dots, k$}
 		\State $\bold{u}_k = \bold{u}_k + h_{j,k} \bold{u}_j - k_{j,k}X \bold{u}_{j}$
 		\EndFor
 		\State $\bold{u}_k =(k_{k+1,k} X-h_{k+1,k}I_{M})^{-1} \bold{u}_k$ \Comment{$I_{M} \in \mathbb{C}^{M \times M}$ is the identity matrix}
 		\EndFor
 		\State $\bold{r} = U_{n+1}\bold{y} $\Comment{The least squares solution}
 		\EndProcedure\end{algorithmic}
 \end{algorithm}	 	
\subsection{Displacement rank and rational Krylov subspaces}
The rational Krylov basis $\{\bold{v},\psi_{1}(Z)\bold{v},\ldots,\psi_{n}(Z)\bold{v}\}$ associated with a rational Krylov subspace $\mathcal{K}_{n+1}(Z,\bold{v};\Xi,\Phi)$ forms the rational Krylov matrix
\begin{equation*}
 		B_{n+1} =\begin{bmatrix}
 			\bold{v} & \psi_{1}(Z)\bold{v} &\dots & \psi_{n}(Z)\bold{v}
 		\end{bmatrix} =W \underbrace{\begin{bmatrix}
			1&\frac{1}{z_1-\xi_1} & \frac{1}{z_1-\xi_2} & \dots & \frac{1}{z_1-\xi_n}\\
			1&\frac{1}{z_2-\xi_1} & \frac{1}{z_2-\xi_2} & \dots & \frac{1}{z_2-\xi_n}\\
			\vdots & \vdots & & \vdots \\
			1&\frac{1}{z_m-\xi_1} & \frac{1}{z_m-\xi_2} & \dots & \frac{1}{z_m-\xi_n}\\
		\end{bmatrix},}_{:=C\in\mathbb{C}^{m\times (n+1)}}
 	\end{equation*}
where $\psi_{i}(Z)=\left(\frac{ I}{Z-\xi_1 I}\right)\prod_{k=1}^{i-1}\left( \frac{Z-\xi_k I}{Z-\xi_{k+1} I} \right)$, $i=1,2,\ldots,n$.

The matrix $C$ exhibits low displacement rank \cite{30}, i.e. for the specific displacement operator based on $Z=diag(\{z_{j}\}_{j=1}^{m})\in \mathbb{C}^{m\times m}$ and $\Psi=diag(\{1,\xi_{1},\xi_{2},\ldots,\xi_{n}\})$, we have
\begin{equation*}
		ZC - C \Psi = \begin{bmatrix}
			z_{1}-1 & 1 & \ldots &1\\
			z_{2}-1 & 1 & \ldots & 1\\
			\vdots & \vdots &  & \vdots\\
			z_{m}-1 & 1 & \ldots & 1
			\end{bmatrix},
	\end{equation*}
	which concludes the matrix $C$ to have a displacement rank of $2$.

Now, we reveal that any rational Krylov matrix associated with a rational Krylov subspace generated by a diagonalizable matrix is related to the matrix $C$ and vice versa. By taking the same approach as Section \ref{2.4}, inserting the eigenvalue decomposition $A= X Z X^{-1}$ in $B_{n+1}$ and assuming $\bold{w}:=X^{-1}\bold{v}$ result in
	\begin{align*}
		B_{n+1} &= \begin{bmatrix}
			\bold{v} & \psi_{1}(A)\bold{v} & \dots &  \psi_{n}(A)\bold{v}\\
		\end{bmatrix}
		= \begin{bmatrix}
			XX^{-1}\bold{v} & \psi_{1}(X Z X^{-1}) \bold{v} & \dots & \psi_{n}(X Z X^{-1})\bold{v}\\
		\end{bmatrix}\\
		&= X\begin{bmatrix}
			\bold{w} & \psi_{1}(Z)\bold{w} & \dots & \psi_{n}(Z) \bold{w}\\
		\end{bmatrix}= X \underbrace{\begin{bmatrix}
				w_1 \\
				& w_2\\
				& & \ddots \\
				& & & w_m
		\end{bmatrix} }_{=:W}
		\underbrace{\begin{bmatrix}
				1 & \psi_{1}(z_1) & \dots & \psi_{n}(z_1)\\
				1 & \psi_{1}(z_2) & \dots & \psi_{n}(z_2)\\
				\vdots & \vdots &  & \vdots\\
				1 & \psi_{1}(z_m) & \dots & \psi_{n}(z_m)\\
		\end{bmatrix}}_{=:C},
	\end{align*}
	and obviously $\textrm{rank}(A B_{n+1} -B_{n+1} \Psi)=2$. Indeed, we can claim that any pair of matrices $A,B_{n+1}$ satisfying this relation are suitable candidates to apply the rational Arnoldi iteration. Moreover, we can utilize the Jordan canonical form $A = X J X^{-1}$ for defective matrices, which is related to the Sobolev rational least squares problem.
	\section{Weighted Sobolev rational least squares problem}\label{s5}
	This section is devoted to the rational Krylov subspace corresponding to the Sobolev rational least squares problem. The related rational Krylov subspace, which involves a Jordan-like matrix, is investigated as well.
	\begin{problem}[Weighted Sobolev rational least squares]\label{prob:wHrLS}
		Suppose that $\{z_j,w_{j}\}_{j=1}^\sigma$ are the given nodes and corresponding weights, respectively. Given the sets of poles $\Xi$ and shifts $\Phi$, and the first $s_{j}$ derivative values $\{f^{(s_{j})}_{j}\}_{j=1}^\sigma$, $s_{j}\leq s$, with $s \geq 0$, construct a rational function $r\in\mathcal{R}_n^\Xi$, $n\leq m-1$, minimizing
		\begin{equation*}
			\sum_{j=1}^\sigma  \sum_{i=1}^{s_{j}}  \vert w_j\vert^2 \left\vert \frac{\prod_{r=1}^{i}\alpha_{r}^{(j)}}{i!} \right\vert^2\vert r^{(i)}(z_j)-f^{(i)}_{j}\vert^2,
		\end{equation*}
	in which $\alpha_{r}^{(j)}$ are non-zero and chosen freely\footnote{$\alpha_{r}^{(j)}$ will appear in the Jordan blocks.}.
	\end{problem}
	Again, we first look at the corresponding rational Krylov subspace for an arbitrary $s$. Let us define the matrix $\mathcal{G}_{j} \in \mathbb{C}^{(s_{j}+1)\times (n+1)}$, $j=1,2,\dots,\sigma$, containing the derivatives of $\{1,\frac{1}{t-\xi_1},\ldots, \frac{1}{t-\xi_n}\}$ up to order $s_{j}$ at $z_{j}$, in the following sense
	\begin{equation*}
	\begin{bmatrix}
	0 & \frac{(-1)^{s_{j}} s_{j}!}{(z_{j}-\xi_{1})^{s_{j}+1}} &\frac{(-1)^{s_{j}}s_{j}!}{(z_{j}-\xi_{2})^{s_{j}+1}}& \dots & \frac{(-1)^{s_{j}}s_{j}!}{(z_{j}-\xi_{n})^{s_{j}+1}}\\
				0 & \frac{(-1)^{s_{j}-1}(s_{j}-1)!}{(z_{j}-\xi_{1})^{s_{j}}} &\frac{(-1)^{s_{j}-1}(s_{j}-1)!}{(z_{j}-\xi_{2})^{s_{j}}}& \dots & \frac{(-1)^{s_{j}-1}(s_{j}-1)!}{(z_{j}-\xi_{n})^{s_{j}}}\\
		\vdots & \vdots & \vdots & & \vdots \\
		0 & \frac{-1}{(z_{j}-\xi_{1})^{2}} &\frac{-1}{(z_{j}-\xi_{2})^{2}}& \dots & \frac{-1}{(z_{j}-\xi_{n})^{2}}\\
				1 & \frac{1}{(z_{j}-\xi_{1})} &\frac{1}{(z_{j}-\xi_{2})}& \dots & \frac{1}{(z_{j}-\xi_{n})}
\end{bmatrix}.
\end{equation*}
By considering $m:=\sigma+\sum_{j=1}^\sigma s_{j}$, and applying the basis $\{1,\frac{1}{t-\xi_1},\ldots, \frac{1}{t-\xi_n}\}$, Problem \ref{prob:wHLS} is equivalent to the following system of equations
\begin{equation*}
	W\underbrace{\left[\begin{array}{ccccc}
			&&\mathcal{G}_{1}&&\\
			&&\mathcal{G}_{2}&&\\
			&&\vdots&&\\
			&&\mathcal{G}_{\sigma}&&\\
		\end{array}\right]}_{=:C^{(c)}\in \mathbb{C}^{m\times (n+1)}} \begin{bmatrix}
		c_0\\
		c_1\\
		\vdots\\
		c_{n}
	\end{bmatrix} \approx W \bold{f},
\end{equation*}
where $W=diag(\bold{w})$, with $\bold{w}$ defined in \eqref{eq26}, and
\begin{equation*}
\bold{f}=\begin{bmatrix}
				f^{(s_{1})}_1~
				f^{(s_{1}-1)}_{1}~
			\dots~
			f_{1}&
			f^{(s_{2})}_2~
			f^{(s_{2}-1)}_{2}~
			\dots~
			f_{2}&
			\dots&
			f^{(s_{\sigma})}_\sigma~
			f^{(s_{\sigma}-1)}_{\sigma}~
			\dots~
			f_\sigma
		\end{bmatrix}^\top \in \mathbb{C}^{m}.
		\end{equation*}
The coefficient matrix $WC^{(c)}$ is equal to the Krylov matrix of the rational Krylov subspace $\mathcal{K}_{n+1}(J,\bold{v};\Xi,\Phi)$ defined by
\begin{equation*}
 		\operatorname{span}\left\{
 			\bold{v} , \frac{ I}{J-\xi_1 I} \bold{v} ,\dots , \left(\frac{ I}{J-\xi_1 I}\right)\prod_{k=1}^{n-1}\left( \frac{J-\xi_k I}{J-\xi_{k+1} I} \right)\bold{v}\right\}.
 	\end{equation*}
Meanwhile, the starting vector is
 \begin{equation*}
	\bold{v}=\begin{bmatrix}      \overbrace{0 \ldots  0}^{s_{1}} & w_{1} &   \overbrace{0 \ldots  0}^{s_{2}} & w_{2} & \dots& \overbrace{0 \ldots  0}^{s_{\sigma}} & w_{\sigma} \end{bmatrix}^\top \in \mathbb{C}^{m},
		\end{equation*}
		and the Jordan-like matrix
	\begin{equation*}
	 J = \left[\begin{array}{cccc}
		J_{s_{1}} &&&\\
		 &J_{s_{2}}&&\\
		 &&\ddots&\\
		&&&J_{s_{\sigma}}\\
	\end{array}\right]\in \mathbb{C}^{m\times m},
	\end{equation*}
	contains $\sigma$ blocks of order $s_{j}+1$ as shown in \eqref{eq25}.
		
		  Notably, if we remove the first column of $C^{(c)}$, it is a confluent Cauchy matrix, and we can take the same strategy.

We dedicate the next subsection to the orthogonality of the Sobolev orthogonal rational functions with respect to a Sobolev inner product. Two algorithms are also given to compute the solution of Problem \ref{prob:wHrLS}.
		\subsection{Rational Arnoldi and Sobolev orthogonal rational functions}\label{sec:Arnoldi_Sobolev}
		To compute an orthonormal basis $Q_{n+1}$ for $\mathcal{K}_{n+1}(J,\bold{v};\Xi,\Phi)$, the rational Arnoldi procedure is utilized. This iteration also gives a Hessenberg pencil containing the recurrence coefficients.
		\begin{theorem}[Krylov induced Sobolev orthogonal rational functions]\label{theorem:KrylovSobOrs}
		Consider the Jordan-like matrix
		\begin{equation*}
	 J = \left[\begin{array}{cccc}
		J_{s_{1}} &&&\\
		 &J_{s_{2}}&&\\
		 &&\ddots&\\
		&&&J_{s_{\sigma}}\\
	\end{array}\right]\in \mathbb{C}^{m\times m},
	\end{equation*}
	with blocks
	\begin{equation*}
		J_{s_{j}} = \begin{bmatrix}
			z_{j}&\alpha_{s_{j}}^{(j)}&&\\
			&z_{j}&&\\
			&&\ddots&\alpha_{1}^{(j)}\\
			&&&z_{j}\\
			
		\end{bmatrix}, \quad j=1,2,\ldots,\sigma,
	\end{equation*}
	and the vector
		\begin{equation*}
	\bold{v}=\begin{bmatrix}      \overbrace{0 \ldots  0}^{s_{1}} & w_{1} &   \overbrace{0 \ldots  0}^{s_{2}} & w_{2} & \dots& \overbrace{0 \ldots  0}^{s_{\sigma}} & w_{\sigma} \end{bmatrix}^\top \in \mathbb{C}^{m}.
		\end{equation*}
Let $\begin{bmatrix}
			\bold{q}_0 & \bold{q}_1 & \dots & \bold{q}_{n}
		\end{bmatrix} = Q_{n+1} \in\mathbb{C}^{m\times (n+1)}$ form a nested orthonormal basis for $\mathcal{K}_{n+1}(J,\bold{v};\Xi,\Phi)$, $n<m$.
		Assume that rational functions $r_k\in\mathcal{R}^\Xi_k$ are defined on the spectrum of $J$ satisfying $\bold{q}_k = r_k(J)\bold{v}$.
		Then $\{r_k\}_{k=0}^{n}$ is the set of Sobolev orthonormal rational functions with respect to the inner product
		\begin{equation}\label{innerRS}
			\langle r_k,r_h\rangle _S = \sum_{j=1}^\sigma  \sum_{i=1}^{s_{j}} \vert w_j\vert^2 \left\vert \frac{\prod_{r=1}^{i}\alpha_{r}^{(j)}}{i!} \right\vert^2 r^{(i)}_{k}(z_j) \overline{r^{(i)}_{h}(z_j)}.
		\end{equation}
	\end{theorem}
	\begin{proof}
The basis vectors satisfy $\bold{q}_k = r_k(J)\bold{v}$, with $r_k\in\mathcal{R}^\Xi_k$.
The orthogonality of the columns of $Q_{n+1}$ implies
\begin{equation}\label{eq5}
	\delta_{kh} = \bold{q}_h^H \bold{q}_k =  \bold{v}^H (r_h(J))^H r_k(J)\bold{v}.
\end{equation}	
Since $r_{k}$ is defined on the spectrum of $J$, by the definition of a matrix function \cite{15}, we know that
\begin{equation}\label{eq9}
	r_k(J) = \left[\begin{array}{cccc}
		r_k(J_{s_{1}}) & &&\\
		& r_k(J_{s_{2}})&&\\
		& & \ddots& \\
		& & & r_k(J_{s_{\sigma}})
	\end{array}\right],
\end{equation}
where
\begin{equation*}
	r_k(J_{s_{j}}) = \begin{bmatrix}
		r_k(z_{j}) & \alpha_{s_{j}}^{(j)}r'_{k}(z_{j})&\dots& \frac{\prod_{r=1}^{s_{j}}\alpha_{r}^{(j)}}{s_{j}!}r^{(s_{j})}_{k}(z_{j})\\
		& r_k(z_{j}) &&\\
		& & \ddots& \alpha_{1}^{(j)} r'_{k}(z_{j})\\
		& & & r_k(z_{j}) \\
	\end{bmatrix}, \quad j=1,2,\ldots,\sigma.
\end{equation*}
Plugging \eqref{eq9} into \eqref{eq5} results in
\begin{equation*}
	\delta_{kh}=\sum_{j=1}^\sigma \sum_{i=0}^{s_{j}} \vert w_j\vert^2 \left\vert \frac{\prod_{r=1}^{i}\alpha_{r}^{(j)}}{i!} \right\vert^2 r^{(i)}_{k}(z_j) \overline{r^{(i)}_{h}(z_j)}.
\end{equation*}	
Thus the rational functions $\{r_k\}_{k=0}^{n}$ are orthonormal with respect to a Sobolev inner product.
		\end{proof}
		Now, the solution $r(t)$ to Problem \ref{prob:wHLS} is subject to find the unknown vector $\bold{y}=\begin{bmatrix}
		y_0 &  \dots & y_n
	\end{bmatrix}^{\top}$
	which is the output of Algorithm \ref{alg:7}.
		\begin{algorithm}[H]
		\caption{The evaluation process of the unknown vector $\bold{y}$}\label{alg:7}
		\begin{algorithmic}[1]
			\State \textbf{Input:} $J \in \mathbb{C}^{m\times m}$, $\Xi = \{\xi_1,\xi_2,\dots, \xi_{n}\}$, $\Phi=\{\phi_1,\phi_2,\dots, \phi_{n}\}$, $\bold{f} \in \mathbb{C}^{m}$, $W \in \mathbb{C}^{m \times m}$, $\bold{v} \in \mathbb{C}^{m}$, the integer $n$
			\State \textbf{Output:} Hessenberg matrices $\underline{H},~\underline{K} \in\mathbb{C}^{(n+1)\times n}$, the unknown vector $\bold{y} \in \mathbb{C}^{n+1}$
			\Procedure{}{}
			\State $\bold{q}_0 = \bold{v}/ \Vert \bold{v}\Vert$\Comment{Rational Arnoldi iteration}
			\For{$k=1,2,\dots,n$}
			\State $\bold{q}_k = (\nu_k J - \mu_k I)^{-1} (\eta_k J - \rho_k I)\bold{q}_{k-1}$ \Comment{With $\frac{\mu_k}{\nu_k} = \xi_k$, and $\frac{\rho_k}{\eta_k} = \phi_k$}
			\For{$j=1,2,\dots, k$}\Comment{Orthogonalization}
			\State $h_{j,k} = \langle \bold{q}_k,\bold{q}_j\rangle_E$
			\State $\bold{q}_k = \bold{q}_k - h_{j,k} \bold{q}_j$
			\EndFor
			\State $h_{k+1,k} = \Vert \bold{q}_k\Vert$
			\State $\bold{q}_k = \bold{q}_k/h_{k+1,k}$ \Comment{Normalization}
			\EndFor
			\State $\underline{K} =  \underline{H}\textrm{diag}(\nu_1,\dots,\nu_n) - I_{n+1}\textrm{diag}(\eta_1,\dots,\eta_n)$\Comment{$I_{n+1}\in \mathbb{C}^{(n+1)\times n}$ is the identity matrix}
			\State $\underline{H} =  \underline{H} \textrm{diag}(\mu_1,\dots,\mu_n) - I_{n+1}\textrm{diag}(\rho_1,\dots,\rho_n)$
			\State $\bold{y} = Q_{n+1}^H W\bold{f}$
			\EndProcedure\end{algorithmic}
	\end{algorithm}	
	At this stage, in order to compute the least squares solution at a given set of sampling points $\{x_j\}_{j=1}^{\tau}$ and its derivatives of order at most $s$, the following algorithm is provided. Notably, in this algorithm, $\bold{S}$ is the column vector of sampling derivatives $\{\tilde{s}_{j}\}_{j=1}^{\tau}$, $\tilde{s}_{j} \leq s$, and $X$ is the Jordan-like matrix associated with the sampling nodes $\{x_{j}\}_{j=1}^{\tau}$. We define the matrix $\mathcal{F}_{j} \in \mathbb{C}^{(\tilde{s}_{j}+1)\times (n+1)}$, $j=1,2,\ldots,\tau$, including the derivatives up to order $\tilde{s}_{j}$ of the Sobolev orthonormal rational functions at $x_{j}$, as
		\begin{equation*}
	\mathcal{F}_{j}:=	\begin{bmatrix}
			0 & r_{1}^{(\tilde{s}_{j})}(x_{j}) & \dots & r_{n}^{(\tilde{s}_{j})}(x_{j})\\
			\vdots & \vdots & & \vdots \\
			0 & r_{1}^{\prime}(x_{j}) & \dots & r_{n}^{\prime}(x_{j})\\
			r_{0}(x_{j}) & r_{1}(x_{j}) & \dots & r_{n}(x_{j})
		\end{bmatrix}.
	\end{equation*}
In Algorithm \ref{alg:8}, the resulting matrix $U_{n+1} \in \mathbb{C}^{M\times (n+1)}$, $M:=\tau+\sum_{j=1}^\tau \tilde{s}_{j}$ is
\begin{equation*}
	U_{n+1}:=\left[\begin{array}{ccccc}
		&&\mathcal{F}_{1}&&\\
		&&\mathcal{F}_{2}&&\\
		&&\vdots&&\\
		&&\mathcal{F}_{\tau}&&\\
	\end{array}\right],
\end{equation*}
		satisfying
		\begin{equation*}
			XU_{n+1}\underline{K} = U_{n+1}\underline{H}.
		\end{equation*}
	Finally, the $M$-vector $\bold{r}:= U_{n+1}\bold{y}$
		\begin{equation*}
		\bold{r}=\begin{bmatrix}     r^{(\tilde{s}_{1})}(x_{1}) \ldots  r^{\prime}(x_{1})& r(x_{1}) &   r^{(\tilde{s}_{2})}(x_{2}) \ldots  r^{\prime}(x_{2})& r(x_{2}) &  & \dots&
				r^{(\tilde{s}_{\tau})}(x_{\tau}) \ldots  r^{\prime}(x_{\tau})& r(x_{\tau})\end{bmatrix}^\top ,
	\end{equation*}
	can be computed which includes the approximations. Moreover, $r_{0}(t)=\frac{1}{\sqrt{w_{1}^{2}+...+w_{m}^{2}}}$ is the Sobolev orthonormal rational function of degree zero directly computed through the inner product \eqref{innerRS}.
	\begin{algorithm}[H]
		\caption{The least squares solution $\bold{r}$}\label{alg:8}
		\begin{algorithmic}[1]
			\State \textbf{Input:} $\bold{y} \in \mathbb{C}^{n+1}$, $\underline{H},~ \underline{K}\in\mathbb{C}^{(n+1)\times n}$, $X \in \mathbb{C}^{M \times M}$, $\tau$-vector $\bold{S}$, $r_{0}(t)$
			\State \textbf{Output:} $U_{n+1} \in \mathbb{C}^{M\times (n+1)}$, $\bold{r} \in \mathbb{C}^{M}$
			\Procedure{}{}
			\State $\bold{u}_0 = \begin{bmatrix}
					\mathcal{F}_{1,1}\\
					\mathcal{F}_{2,1}\\
				\vdots\\
					\mathcal{F}_{\tau,1}
			\end{bmatrix}$\Comment{	$\mathcal{F}_{j,1}$ is the first column of the matrix $\mathcal{F}_{j}$, for $j=1,2,\ldots,\tau$}
			\For{$k=1,2,\dots,n$}
			\State $\bold{u}_k = 0$
			\For{$j=1,2,\dots, k$}
			\State $\bold{u}_k = \bold{u}_k + k_{j,k}X \bold{u}_{j}-h_{j,k} \bold{u}_j$
			\EndFor
			\State $\bold{u}_k =(h_{k+1,k}I_{M}-k_{k+1,k}X)^{-1} \bold{u}_k$\Comment{$I_{M}\in \mathbb{C}^{M\times M}$ is the identity matrix}
			\EndFor
			\State $\bold{r} = U_{n+1} \bold{y}$\Comment{The least squares solution}
			\EndProcedure\end{algorithmic}
	\end{algorithm}	
	\section{Numerical examples}	\label{s6}
	In this section, to confirm the accuracy of all presented approaches, the numerical results of various examples are computed. For this purpose, we organize this section as follows:
\begin{description}
\item[$\bullet$]
  The numerical errors are reported to monitor the computational performance of the suggested methods. In all of the examples, the errors are calculated through the $L^{\infty}$-norm over an arbitrary set of sampling nodes $\{x_{j}\}_{j=1}^{M}$.

  \item[$\bullet$]
  We execute a code for the Arnoldi and rational Arnoldi algorithm with re-orthogonalization, i.e., the algorithm where the orthogonalization of a current vector against previously computed set is performed exactly twice, which is called "twice is enough" re-orthogonalization algorithm and introduced by Kahan and Parlett \cite{Parlett}.

\item[$\bullet$] To approve the superiority of the approach of  Arnoldi and rational Arnoldi, we compare our results with those obtained without Arnoldi.

\item[$\bullet$]  To illustrate the stability of the approach, the performance of the strategies is investigated for large $n$.
\end{description}
The process of calculations is done via Matlab R2022a, running on a computer system with an AMD Ryzen 7 pro 6850u CPU with 16 GiB memory.
\begin{example}\label{ex2}
\cite{10,11} Let $f(t) = \frac{1}{1+25t^{2}}$ on $[-1,1]$, and consider a Sobolev polynomial least squares problem, where the information available are $\{f_j^{(s_{j})}\}_{j=1}^{\sigma}$, $s_{j} \leq 2$, and $s_{j}$ are chosen randomly.
\end{example}
The problem is solved implementing the proposed method using $\sigma=2n+1$ first kind Chebyshev-Gauss quadrature nodes, and Legendre-Gauss quadrature nodes along with corresponding weights as the node-weight pairs $\{z_{j},w_{j}\}_{j=1}^{\sigma}$. Table \ref{tab2} reports the obtained numerical results for both $f$ and its derivatives.
\begin{table}[!h]
\begin{center}
\small{
\caption{The errors of Example \ref{ex2} using Arnoldi for different pairs of node-weight and $n$.\label{tab2}}\vspace{-0.2cm}
\setlength\tabcolsep{3.5pt}
\begin{tabular}{@{}c|cc|cc|cc@{}}
\hline
\hline
\multicolumn{3}{c|}{$ ||f-p||_{\infty}$} &\multicolumn{2}{c|}{$||f'-p'||_{\infty}$} &\multicolumn{2}{c}{$||f''-p''||_{\infty}$}\\
  [0.1 cm]
\cline{1-7}\\[-0.3 cm]
 $n$&  $\textit{Chebyshev} $ & $\textit{Legendre}$&$\textit{Chebyshev} $ & $\textit{Legendre}  $&$\textit{Chebyshev} $ & $\textit{Legendre} $\\
 \hline
 \hline
 30 &$9.31\times 10^{-2}$&  $8.54\times 10^{-2}$&$-$ & $-$ &$-$ & $-$\\
 60  &$3.02 \times 10^{-4}$& $3.35\times 10^{-4}$ & $5.40\times 10^{-3}$&$6.30 \times 10^{-3}$& $3.32\times 10^{-1}$ & $-$ \\
120&$1.81\times 10^{-9}$& $1.68\times 10^{-9}$ & $8.88\times 10^{-8}$&$2.84\times 10^{-7}$& $5.90\times 10^{-6}$ & $8.57\times 10^{-4}$\\
240 &$6.07 \times 10^{-14}$& $2.46\times 10^{-14}$ & $2.99\times 10^{-12}$&$1.70 \times 10^{-12}$& $3.48\times 10^{-9}$ & $1.50\times 10^{-8}$\\
\hline
\hline
\end{tabular}}\end{center}\vspace{-0.5cm}
\end{table}

The problem is also solved via the Arnoldi iteration with re-orthogonalization, and the errors are presented in Table \ref{tab10}. One can be deduced from Table \ref{tab2} and \ref{tab10} that the re-orthogonalization process increases the accuracy by one or two digits.
\begin{table}[!h]
	\begin{center}
		\small{
			\caption{The errors of Example \ref{ex2} using Arnoldi with re-orthogonalization for different pairs of node-weight and $n$.\label{tab10}}\vspace{-0.2cm}
			\setlength\tabcolsep{3.5pt}
			\begin{tabular}{@{}c|cc|cc|cc@{}}
				\hline
				\hline
				\multicolumn{3}{c|}{$ ||f-p||_{\infty}$} &\multicolumn{2}{c|}{$||f'-p'||_{\infty}$} &\multicolumn{2}{c}{$||f''-p''||_{\infty}$}\\
				[0.1 cm]
				\cline{1-7}\\[-0.3 cm]
				$n$&  $\textit{Chebyshev} $ & $\textit{Legendre}$&$\textit{Chebyshev} $ & $\textit{Legendre}  $&$\textit{Chebyshev} $ & $\textit{Legendre} $\\
				\hline
				\hline
				30 &$2.90\times 10^{-3}$&  $3.95\times 10^{-2}$&$3.37\times 10^{-1}$& $6.12\times 10^{-1}$ &$-$ & $-$\\
				60  &$7.30\times 10^{-5}$& $7.87\times 10^{-5}$ & $2.00\times 10^{-3}$&$2.20 \times 10^{-3}$& $1.06\times 10^{-1}$ & $3.42\times 10^{-1}$ \\
				120&$7.08\times 10^{-10}$& $1.34\times 10^{-9}$ & $2.79\times 10^{-8}$&$4.57\times 10^{-8}$& $2.79\times 10^{-8}$ & $1.66\times 10^{-5}$\\
				240 &$2.55 \times 10^{-15}$& $2.00\times 10^{-15}$ & $1.91\times 10^{-14}$&$2.86 \times 10^{-13}$& $1.28\times 10^{-10}$ & $4.59\times 10^{-9}$\\
				\hline
				\hline
	\end{tabular}}\end{center}\vspace{-0.5cm}
\end{table}

Undeniably, the results given in Table \ref{tab10} indicate the high accuracy of the presented scheme due to the regular decrease in the errors, in particular for large $n$.
The semi-log representation of the numerical errors in the case of first kind Chebyshev nodes and re-orthogonalization process is also depicted in Figure \ref{fg2}.
\begin{figure}[!ht]
\centerline{\includegraphics[width=8.3cm]{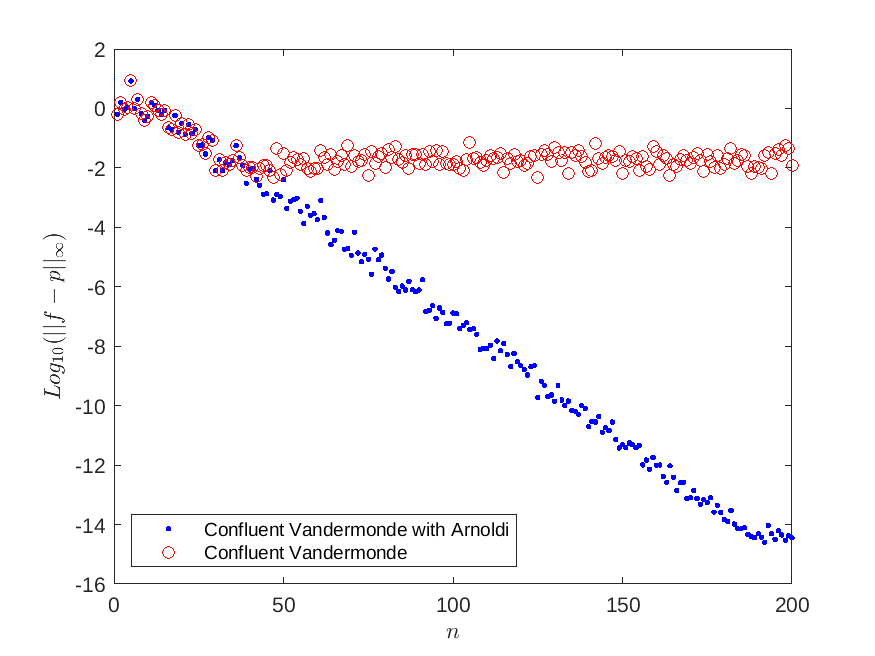},\includegraphics[width=8.3cm]{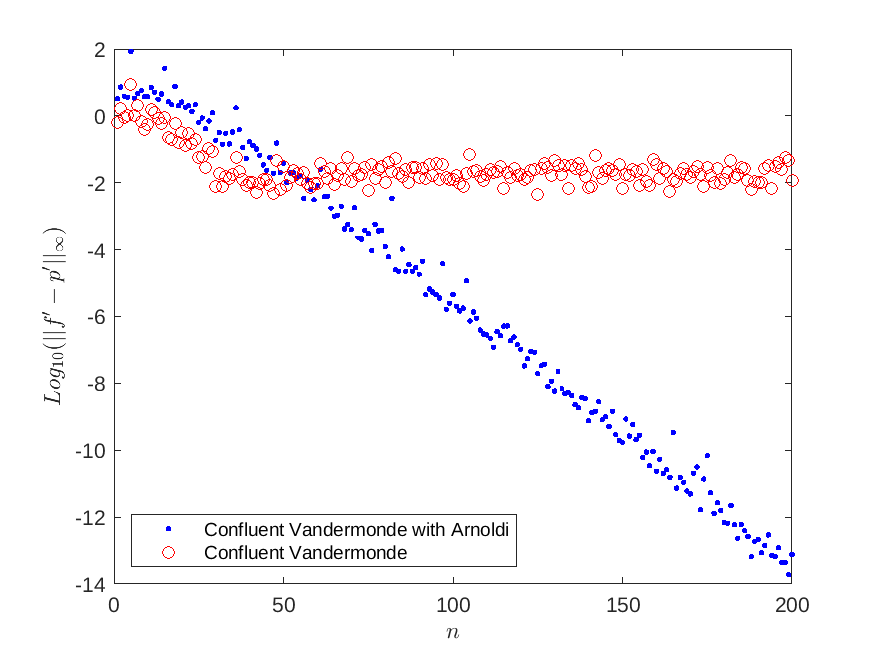}}
\centerline{\includegraphics[width=8.3cm]{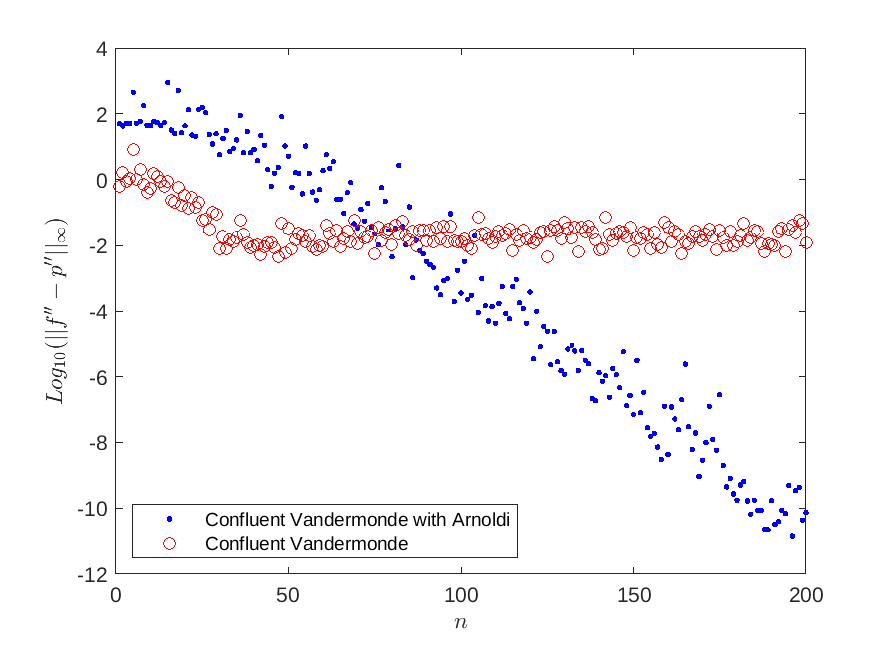}}
\caption{Semi-log representation of the errors for Example \ref{ex2} using Arnoldi with re-orthogonalization on first kind Chebyshev nodes.}\label{fg2}
\end{figure}

The example is solved via Vandermonde without Arnoldi as well, and as it can be deduced from Figure \ref{fg2}, Vandermonde with Arnoldi gives us highly accurate approximations for $f,~f'$, and $f''$, while without Arnoldi we just achieve $2$ digits of accuracy.
\begin{example}\label{ex3}
Let us consider $f(t) = |t|$ on $[-1,1]$. The approximation is the solution of the rational least squares problem. $2000$ exponentially clustered nodes around zero along with weights $1$ are considered as the node-weight pairs. The poles are chosen as the complex numbers $\xi_{j}=\pm \sqrt{|\delta_{j}|}i$, $j=1,2,\ldots,n$, with so-called tapered exponentially clustered poles \cite{28,29}
	\[\delta_{j}=-2 \exp{( -\sqrt{2}\pi(\sqrt{n}-\sqrt{j}))},\quad j=1,2,\ldots,n.\]
\end{example}
We evaluate the approximation taking the strategy of Arnoldi and provide the attained results in Table \ref{tab3} and Figure \ref{fg3}.
\begin{table}[!h]
\begin{center}
\small{
\caption{The errors of Example \ref{ex3} for different $n$.\label{tab3}}\vspace{-0.2cm}
\setlength\tabcolsep{3.5pt}
\begin{tabular}{@{}c|c@{}}
\hline
\hline
 $n$&   $||f-r||_{\infty}$ \\
 \hline
 \hline
 \\[-8pt]
 15 &$4.44 \times 10^{-5}$\\
 30 &$1.27 \times 10^{-6}$\\
60&$8.23 \times 10^{-9}$ \\
120 &$2.71 \times 10^{-9}$\\
\hline
\hline
\end{tabular}}\end{center}\vspace{-0.5cm}
\end{table}
\begin{figure}[!ht]
	\centerline{\includegraphics[width=9cm,height=7cm]{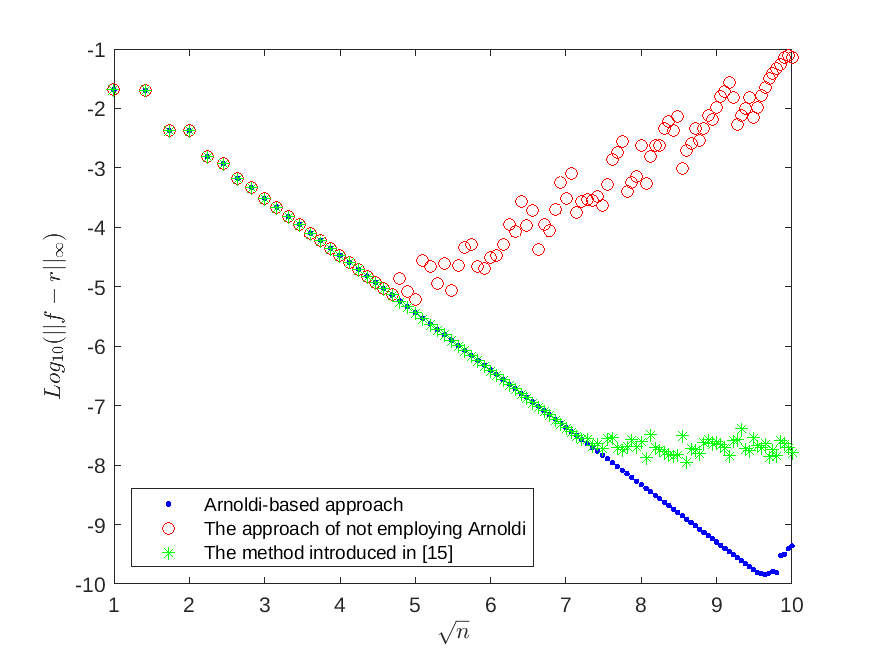}}
	\caption{Semi-log representation of the errors for Example \ref{ex3}.}\label{fg3}
\end{figure}

Regardless of the less smoothness of the function, we arrive at a regular error reduction which is a valid evidence of a strong performance of this approach. Moreover, one can conclude from Figure \ref{fg3} that solving the system \eqref{Cauchy system} directly is not a stable approach due to the dramatic increase in the errors for $n>60$. We also apply the basis $\{1,\frac{\xi_1}{t-\xi_1},\ldots, \frac{\xi_n}{t-\xi_n}\}$ introduced in \cite{28} instead of the basis $\{1,\frac{1}{t-\xi_1},\ldots, \frac{1}{t-\xi_n}\}$ and solve the relevant system of equations directly\footnote{A special case of \cite{28}, with polynomial part restricted to degree $0$, is considered.}. As it is shown in Figure \ref{fg5}, the rational orthogonal basis leads to a bit more accurate results.

It is noteworthy that we also implemented the rational Arnoldi procedure with re-orthogonalization for this example. However, the results are not presented since the accuracy remained unchanged by this process.
\begin{example}\label{ex1}
Consider the function of interest $f(t) = \sqrt{t}$ on $(0,1]$ which we will approximate using $2000$ exponentially clustered nodes around zero along with weights $1$. Tapered exponentially clustered poles \cite{28,29}
	\[\xi_{j}=-2 \exp{( -\sqrt{2}\pi(\sqrt{n}-\sqrt{j}))},\quad j=1,2,\ldots,n,\]
	are deemed as the set of poles.
\end{example}
For this example, the Arnoldi-based approach is implemented, and the numerical errors are reported in Table \ref{tab1}. To make a comparison with the approach of not employing Arnoldi and the method introduced in \cite{28}, we can refer to Figure \ref{fg1}. Although the rational orthogonal basis addresses us toward a better result in the previous example, we can observe that our approach is less stable than the approach of \cite{28} in this example, according to Figure \ref{fg1}.
\begin{table}[!h]
	\begin{center}
		\small{
			\caption{The errors of Example \ref{ex1} for different $n$.\label{tab1}}\vspace{-0.2cm}
			\setlength\tabcolsep{3.5pt}
			\begin{tabular}{@{}c|c@{}}
				\hline
				\hline
				$n$&   $||f-r||_{\infty} $ \\
				\hline
				\hline
				\\[-8pt]
				15 &$2.71 \times 10^{-4}$\\
				30 &$7.19 \times 10^{-6}$\\
				60&$2.29 \times 10^{-6}$ \\
				120 &$2.40 \times 10^{-2}$\\
				\hline
				\hline
	\end{tabular}}\end{center}\vspace{-0.5cm}
\end{table}
\begin{figure}[!ht]
	\centerline{\includegraphics[width=9cm,height=7cm]{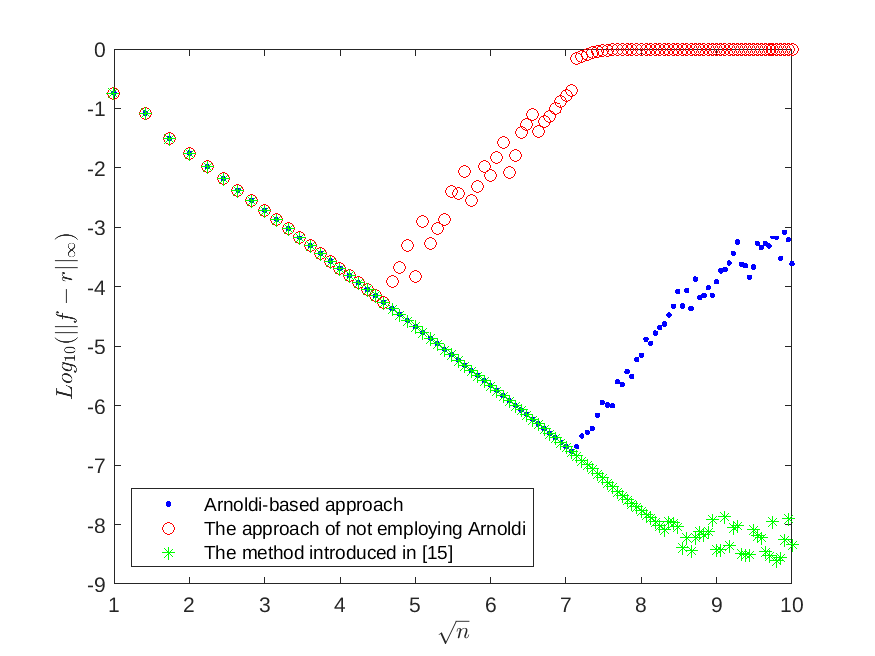}}
	\caption{Semi-log representation of the errors for Example \ref{ex1}.}\label{fg1}
\end{figure}
\begin{example}\label{ex5}
 To assess the method introduced for the Sobolev rational least squares problem, let $f(t) =t\sqrt{t}$ on $(0,1]$, where the information available are $\{f_j^{(s_{j})}\}_{j=1}^{\sigma}$, $s_{j} \leq 1$, and $s_{j}$ are chosen randomly.
\end{example}
Considering the same nodes and poles as the previous example, we tackle this problem based on the Sobolev orthogonal rational functions. The numerical results are depicted through Table \ref{tab5} and Figure \ref{fg5}.
\begin{table}[!h]
	\begin{center}
		\small{
			\caption{The errors of Example \ref{ex5} for different $n$.\label{tab5}}\vspace{-0.2cm}
			\setlength\tabcolsep{6.5pt}
			\begin{tabular}{@{}c|cc@{}}
				\hline
				\hline
				$n$&  $||f-r||_{\infty} $ & $||f'-r'||_{\infty}$\\
				\hline
				\hline
				10 &$4.30\times 10^{-3}$&  $4.95\times 10^{-2}$\\
				20  &$2.39 \times 10^{-4}$& $3.30\times 10^{-3}$  \\
				40&$6.56\times 10^{-6}$& $9.11\times 10^{-5}$\\
				80 &$5.83 \times 10^{-8}$& $3.57\times 10^{-7}$\\
				\hline
				\hline
	\end{tabular}}\end{center}\vspace{-0.5cm}
\end{table}
\begin{figure}[!ht]
	\centerline{\includegraphics[width=8.3cm]{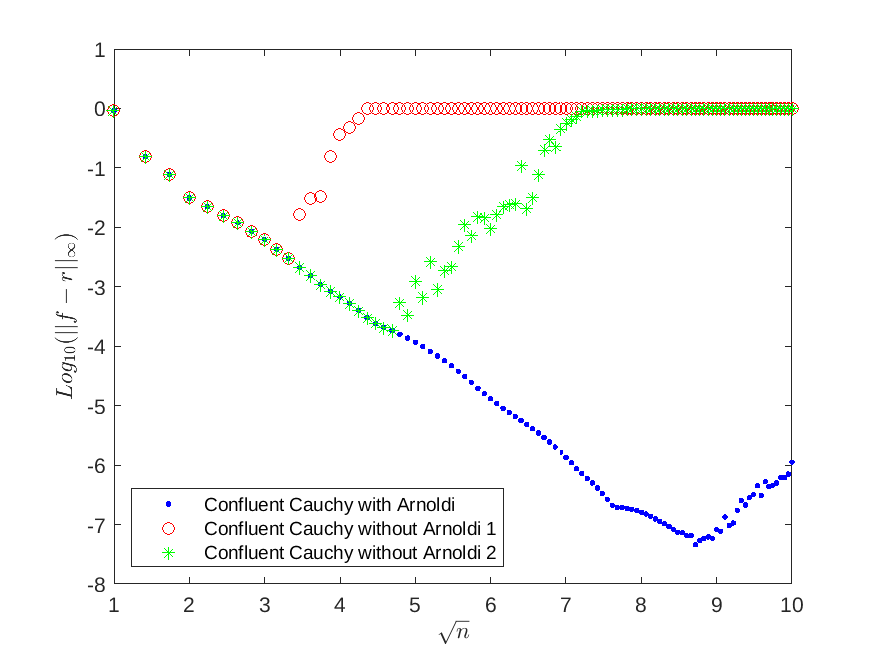},\includegraphics[width=8.3cm]{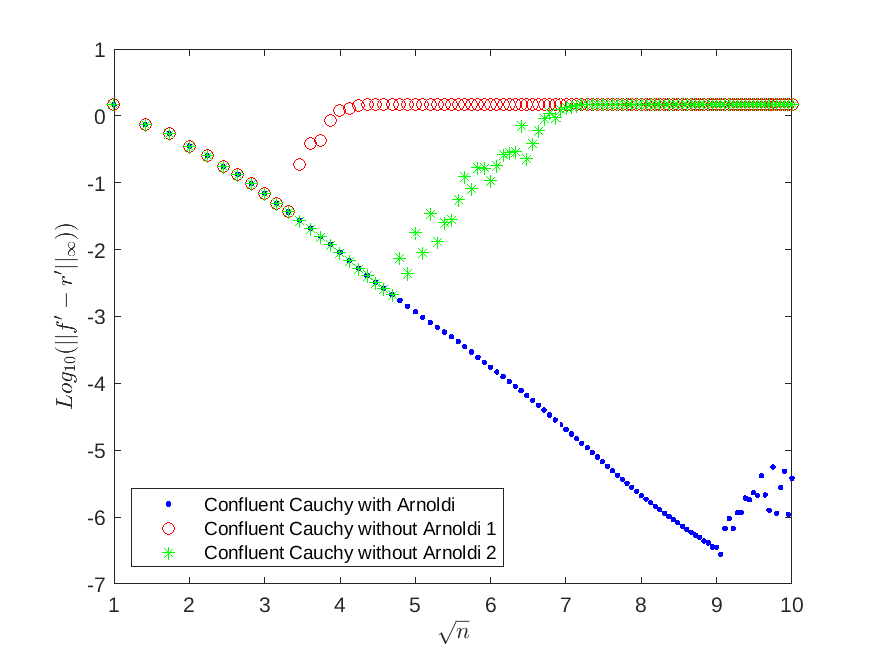}}
	\caption{Semi-log representation of the errors for Example \ref{ex5}.}\label{fg5}
\end{figure}

In view of the presented results in Table \ref{tab5}, seven and six digits of accuracy is achieved for the approximation of $r$ and $r'$ respectively, regardless of the singularity behavior of $f$. In addition, Figure \ref{fg5} leads us to conclude that the methods based on both $\{1,\frac{1}{t-\xi_1},\ldots, \frac{1}{t-\xi_n}\}$ and $\{1,\frac{\xi_1}{t-\xi_1},\ldots, \frac{\xi_n}{t-\xi_n}\}$ are unstable for this example, while the Sobolev rational basis works stably.
\section{Conclusion}
In this article, we reformulated the problem of the least squares fitting with an ill-conditioned basis to a problem based on orthogonal polynomials, Sobolev orthogonal polynomials, orthogonal rational functions, and Sobolev orthogonal rational functions. For this purpose, the connection between these bases and the orthonormal bases for the Krylov subspaces has been established which allowed us to use the Arnoldi procedure. The comparison of the Arnoldi-based approach and the approach of not employing Arnoldi has been made on each example of the last section.
\section*{Funding}
The research was partially supported by the Research Council KU Leuven (Belgium), project $C16/21/002$ (Manifactor: Factor Analysis for Maps into Manifolds) and by the Fund for Scientific Research -- Flanders (Belgium), projects $G0A9923N$ (Low rank tensor approximation techniques for up- and downdating of massive online time series clustering) and G0B0123N (Short recurrence relations for rational Krylov and orthogonal rational functions inspired by modified moments).
\section*{Declarations}
\textbf{Code and data availability} All code and numerical experiments are made publicly available on the Numa homepage \url{https://numa.cs.kuleuven.be/} under the software section, related to the Momentum project.

	\clearpage

\end{document}